\documentclass[a4paper,11pt]{article}

\usepackage[english]{babel}
\usepackage{amsmath,amsthm,amssymb,thmtools,enumitem,mathrsfs,mathtools,url,doi,titlesec,bm,graphicx}
\usepackage[left=2.5cm,right=2.5cm,top=2.5cm,bottom=2.5cm,bindingoffset=0cm]{geometry}
\usepackage[titletoc,title]{appendix}
\usepackage{longtable,booktabs}
\usepackage[table]{xcolor}

\setlist[enumerate]{topsep=0pt,label=\textup{(\arabic*)},leftmargin=\parindent,labelsep=.5em}
\setlist{noitemsep}

\usepackage{quoting}
\quotingsetup{vskip=0pt}

\usepackage[T1]{fontenc}

\usepackage{tikz}
\usetikzlibrary{arrows,matrix}
\usepackage{tikz-cd}

\theoremstyle{plain}
\declaretheorem[numberlike=subsection]{proposition}
\declaretheorem[numberlike=subsection]{theorem}

\declaretheorem[numberlike=subsection]{question}
\declaretheorem[numbered=no,name=Theorem A]{ThmA}
\declaretheorem[numbered=no,name=Theorem B]{ThmB}

\theoremstyle{definition}

\declaretheorem[numberlike=subsection]{example}
\declaretheorem[numberlike=subsection]{remark}

\declaretheorem[numberlike=subsection]{notation}
\declaretheorem[numbered=no,name=Acknowledgement]{acknowl}

\numberwithin{equation}{subsection}

\titleformat{\section}[block]
  {\filcenter\normalfont\large\bfseries}{\thesection.}{.5em}{}
\titleformat{\subsection}[runin]
  {\normalfont\bfseries}{\thesubsection}{.5em}{}

\titlespacing*{\section}{0pt}{6ex plus 1ex minus .2ex}{3ex plus .2ex}
\titlespacing*{\subsection}{0pt}{\topsep}{.5em}

\input{BMmacsLatex.sty}

\begin{document}

\begin{center}
\textbf{\Large THE DELIGNE-MOSTOW LIST}
\vspace{4mm}

\textbf{\Large AND SPECIAL FAMILIES OF SURFACES}
\vspace{8mm}

\textit{by}
\bigskip

{\Large Ben Moonen}
\end{center}
\vspace{6mm}

{\small 

\noindent
\begin{quoting}
\textbf{Abstract.} We study whether there exist infinitely many surfaces with given discrete invariants for which the $H^2$ is of CM type. This is a surface analogue of a conjecture of Coleman about curves. We construct a large number of examples of families of surfaces with $p_\geom = 1$ that in the moduli space cut out a special subvariety; these provide a positive answer to our question. Most of these are families of K3 surfaces but we also obtain some families of surfaces of general type. As input for our construction we use the work of Deligne and Mostow. Finally we prove that a very general K3 surface cannot be dominated by a product of curves of small genus.
\medskip

\noindent
\textit{AMS 2010 Mathematics Subject Classification:\/} 14D, 14G35, 14J10
\end{quoting}

} 
\vspace{6mm}

\section{Introduction}

This paper is motivated by the following question.

\begin{question}
\label{question:A}
Given $q$ ($= h^{0,1}$) and~$K^2$, do there exist infinitely many surfaces~$S$ with these invariants and with $p_\geom(S) =1$ such that $H^2(S,\mQ)$ is a Hodge structure of CM type?
\end{question}

As we shall discuss, this may be viewed as a surface analogue of a well-known conjecture of Coleman about Jacobians of curves. (See for instance~\cite{MO}.) Guided by experience from the curve case, our main result is the construction of one hundred and fifty connected families of surfaces $F \colon S \to M$ with $\dim(M) > 0$ such that the parameter values $\lambda \in M$ for which $H^2(S_\lambda,\mQ)$ is of CM type lie analytically dense in~$M$.

The examples we find are non-singular models of product-quotients $(C_\lambda \times D)/G$, where 
\begin{itemize}
\item $C_\lambda$ is an $(N-3)$-dimensional family of $m$-cyclic covers of~$\mP^1$ with $N$ branch points; 
\item $D$ is a fixed $m$-cyclic cover of~$\mP^1$ whose Jacobian is of CM type;
\item $G \cong \mZ/m\mZ$ acts diagonally on $C_\lambda \times D$.
\end{itemize}
As we shall explain in Section~\ref{sec:ExaDMList}, such families may be given by specifying, in addition to the numbers~$m$ and~$N$, the ramification data for the covers $C_\lambda \to \mP^1$ and $D \to \mP^1$ and then varying the branch points of the first map. In total this leads us to consider data $(m,N,a,b)$ satisfying suitable conditions. To such data we associate a family of surfaces $F \colon S \to M$ with $p_\geom = 1$ and irregularity $q=0$, and the variation of Hodge structure $R^2F_* \mQ_S(1)$ gives rise to a period morphism $\theta \colon M \to \Sh_\mK(\SO(\phi),X\bigr)$ into an orthogonal Shimura variety.

The first main result of the paper may then be summarized as follows. The notion of a special subvariety of a Shimura variety is the one discussed, for instance, in~\cite{MO}.

\begin{ThmA}
For each of the one hundred and fifty $4$-tuples $(m,N,a,b)$ listed in\/~\textup{Table~\ref{table:mNab}}, the image of the associated period morphism $\theta \colon M \to \Sh_\mK(\SO(\phi),X\bigr)$ is dense in an $(N-3)$-dimensional special subvariety. Consequently, the set of parameter values $\lambda \in M$ for which $H^2(S_\lambda,\mQ)$ is of CM type is analytically dense in~$M$. In six of these families the surfaces~$S_\lambda$ are of general type, with $K^2=1$ in four families, and with $K^2 = 2$ and $K^2 = 3$ each occurring in one family. The remaining one hundred and forty-four examples are families of K3 surfaces. 
\end{ThmA}

Four further examples, including one family of surfaces of general type with $K^2=1$, are obtained by a variant of this construction.

\subsection{}
The main point of this paper is certainly not that we have new construction techniques for surfaces. Especially for the K3 examples we obtain, some of the ideas involved can already be found in \cite{Kondo}, \S~3, and many of the K3 examples are among those studied in~\cite{GarbPeneg}. Our contribution, if any, lies in the way we use ``twisting'' (in the sense explained in Section~\ref{sec:Twisting}) to obtain a connection between:
\begin{itemize}
\item[(a)] families of curves for which some isogeny factor of the Jacobians give rise to special a special subvariety in~$A_g$;
\item[(b)] special families of surfaces with $p_\geom = 1$ that provide a positive answer to Question~\ref{question:A}. 
\end{itemize}
With a couple of exceptions, the families of curves in~(a) that we take are those occurring in the work of Deligne and Mostow~\cite{DelMost}.

\subsection{}
It is clear from the way we construct the examples, that the surfaces~$S$ we obtain are dominated by a product of curves. It is an open problem whether the same is true for arbitrary complex K3 surfaces. While we are unable to answer this, our second main result implies that we cannot hope to dominate a very general K3 surface by a product of curves of small genus.

\begin{ThmB}
Suppose $S$ is a very general K3 surface that is dominated by a product of curves $C \times D$. Then the Jacobians $J_C$ and~$J_D$ both contain the simple Kuga-Satake partner of~$S$ as an isogeny factor. In particular, $C$ and~$D$ both have genus at least~$512$.
\end{ThmB}

We refer to Section~\ref{sec:DPC}, in particular Theorem~\ref{thm:K3Thm}, for a more explicit statement of what ``very general'' amounts to, and for what we mean by the simple Kuga-Satake partner.

\begin{notation}
\label{ssec:notat}
(a) By a Hodge structure of K3 type we mean a polarizable $\mQ$-Hodge structure of type $(-1,1) + (0,0) + (1,-1)$ with Hodge numbers $1,n,1$ for some~$n$. By a VHS of K3 type over some base variety~$S$ we mean a polarizable variation of Hodge structure whose fibers are of K3 type.

(b) Let $E$ be a CM field with totally real subfield~$E_0$. If $V$ is an $E$-vector space equipped with a hermitian form~$\psi$, we denote by $\UU_E(V,\psi)$ the associated unitary group, viewed as an algebraic group over~$\mQ$. More formally, $\UU_E(V,\psi) = \Res_{E_0/\mQ}\, \UU(V,\psi)$. Similarly we write $\GL_E(V)$ for $\Res_{E/\mQ}\, \GL_E(V)$.

(c) If $x$ is a rational number, $\langle x\rangle$ denotes its fractional part; e.g., $\langle -\frac{3}{5}\rangle = \frac{2}{5}$.

(d) We write $\QHS$ for the category of polarizable $\mQ$-Hodge structures.
\end{notation}

\begin{acknowl} 
I thank Johan Commelin for inspiring discussions about the topic of this paper and for his help with programming in python.
\end{acknowl}

\section{A surface analogue of a conjecture of Coleman}
\label{sec:SurfColeman}

\subsection{}
\label{ssec:famsurfS/B}
Let $F \colon S \to M$ be a projective family of surfaces with $p_\geom = 1$ over a smooth connected base variety~$M$. We consider the associated variation of Hodge structure $R^2F_* \mQ_S(1)$, which, by our assumption that the fibres have geometric genus~$1$, is of K3 type. Let $0 \in M$ be a base point, $V = H^2(S_0,\mQ)\bigl(1\bigr)$, and let $V_\mZ \subset V$ be the lattice given by $H^2(S_0,\mZ)\bigl(1\bigr)$ modulo torsion. The choice of a relatively ample bundle on~$S$ gives rise to a polarization form $\phi \colon V \otimes V \to \mQ$.

Let $h \colon \mS \to \SO(\phi)_\mR$ (with $\mS$ the Deligne torus) be the homomorphism that defines the Hodge structure on~$V$, and let $X$ be the $\SO(\phi)\bigl(\mR\bigr)$-conjugacy class of~$h$. The pair $\bigl(\SO(\phi),X\bigr)$ is a Shimura datum, which, up to isomorphism, is independent of the choice of base point in~$M$. Let $\mK \subset \SO(\phi)\bigl(\Af\bigr)$ be the compact open subgroup of elements that stabilize the lattice $V_\mZ \otimes \Zhat$, and let $\Sh_\mK = \Sh_\mK\bigl(\SO(\phi),X\bigr)$ be the associated (complex) Shimura variety. (See the next remark.) The tautological representation $\SO(\phi) \to \GL(V)$ gives rise to a variation of Hodge structure~$\cV$ over~$\Sh_\mK$, and there is a well-defined period morphism
\[
\theta \colon M \to \Sh_\mK\bigl(\SO(\phi),X\bigr)
\]
such that $R^2F_*\mQ_S(1)$ is the pull-back of~$\cV$ via~$\theta$.

\begin{remark}
\label{rem:ShKStack}
The Shimura variety $\Sh_\mK$ is in fact a stack. But the stack structure is of the simplest possible kind: $\Sh_\mK$ is the quotient of a non-singular variety by a finite group. Readers who prefer to avoid stacks may, throughout the discussion, pass to the finite cover of~$\Sh_\mK$ given by a neat subgroup of~$\mK$. For the results we will discuss, it makes no difference.
\end{remark}

\begin{question}
\label{question:B}
Given $q$ and~$K^2$, do there exist families $S \to M$ of surfaces with $p_\geom = 1$ as in~\textup{\ref{ssec:famsurfS/B}} with the given values of~$q$ and~$K^2$, such that
\begin{enumerate}[label=\textup{(\alph*)}] 
\item the closure of $\theta(M)$ in $\Sh_K\bigl(\SO(\phi),X\bigr)$ is a special subvariety;
\item either $\theta(M)$ has positive dimension or the family $S \to M$ is not isotrivial?
\end{enumerate}
\end{question}

\subsection{}
Questions \ref{question:A} and~\ref{question:B} are equivalent. As on any special subvariety in a Shimura variety the special points lie analytically dense, a positive answer to~\ref{question:B} implies a positive answer to~\ref{question:A}. The converse follows from the fact that surfaces with $p_\geom = 1$ form a bounded family, together with the Andr\'e-Oort conjecture, which for Shimura varieties $\Sh_\mK\bigl(\SO(\phi),X\bigr)$ follows from the main result of Tsimerman in~\cite{Tsim}. (This builds upon work of many other people and uses that the Shimura variety associated with the group $\CSpin(\phi)$ is of Hodge type.)

\begin{example}
As is well-known, if for $F \colon S \to M$ we take a universal family of polarised K3 surfaces, the associated period morphism~$\theta$ is dominant. We refer to \cite{Rizov} for a detailed discussion of the relation between the moduli of polarised K3 surfaces and Shimura varieties. It follows that Question~\ref{question:A} has a positive answer for K3 surfaces. The examples we shall discuss include many explicit families of K3 surfaces for which the image in $\Sh_\mK\bigl(\SO(\phi),X\bigr)$ under the period morphism is dense in a special subvariety of positive dimension.
\end{example}

\begin{example}
Consider surfaces with $p_\geom=q=1$ and $K^2=3$. As shown in~\cite{CatCil}, there is a $5$-dimensional connected moduli space of such surfaces for which the general albanese fibre is a curve of genus~$3$. The surfaces of this type with given albanese~$A$ are described as divisors in the third symmetric power~$A^{(3)}$ (which is a $\mP^2$-bundle over~$A$) in a suitable linear system $|4D-F|$. We refer to~\cite{CatCil} for details. By \cite{PolizziK2=3}, Corollary~6.16 and Proposition~6.18, for a given elliptic curve~$A$, there exist a unique such surface~$S$ with $A$ as albanese for which the Picard number equals $h^{1,1}(S) = 9$. For these the transcendental part of $H^2(S,\mQ)\bigl(1\bigr)$ has Hodge numbers $1,0,1$ and is therefore of CM type. Varying~$A$ we obtain a $1$-parameter family of surfaces that provides a positive answer to Question~\ref{question:A}. The period map for this family is constant.
\end{example}

\begin{remark}
In what we have discussed so far, one may drop the assumption that $p_\geom =1$. See also Section~\ref{sec:Further}. The main reason for us to make this assumption is that for surfaces with $p_\geom = 1$ the period domain for the $H^2$ is hermitian symmetric, so that we obtain a period morphism with values in a Shimura variety.  
\end{remark}

\section{Twisting Hodge structures with action by a CM field}
\label{sec:Twisting}

We discuss a technique to relate Hodge structures of K3 type to abelian varieties. This involves what van Geemen~\cite{vGeemen} calls a ``half-twist''. Our interpretation of this notion is phrased purely in terms of a tensor product over a coefficient field, which makes it easier to apply and to extend it to other settings. The ideas in this section have also been exploited in our paper~\cite{TMTCC1} in a motivic context.

\subsection{}
Let $E$ be a CM field, $E_0 \subset E$ the maximal totally real subfield. Write $\Sigma$ for the sets of complex embeddings of~$E$.

Let $\QHS_{(E)}$ be the category of $E$-modules in~$\QHS$. For $V \in \QHS_{(E)}$, let $V(\sigma)\subset V_\mC$ denote the subspace on which $E$ acts through the embedding~$\sigma$, and let $V(\sigma)^{p,q} = V(\sigma) \cap V_\mC^{p,q}$. This gives a decomposition
\[
V_\mC = \bigoplus_{\sigma \in \Sigma} \bigoplus_{p,q}\; V(\sigma)^{p,q}\qquad \text{with}\quad 
\overline{V(\sigma)^{p,q}}  = V(\bar{\sigma})^{q,p}\, .
\]
We call the function $h_V \colon \Sigma \times \mZ^2 \to \mN$ given by $h_V(\sigma;p,q) = \dim V(\sigma)^{p,q}$ the multiplication type of~$V$.

The category $\QHS_{(E)}$ is an $E$-linear Tannakian category. If $V$, $W$ are objects of~$\QHS_{(E)}$, the Hodge decomposition of $V \otimes_E W$ is given by the rule that $\bigl(V\otimes_E W\bigr)(\sigma) = V(\sigma) \otimes_\mC W(\sigma)$ as bigraded vector spaces, for all $\sigma \in \Sigma$. Hence the multiplication type of $V \otimes_E W$ is given by
\[
h_{V\otimes_E W}(\sigma;p,q) = \sum_{\substack{p_1+p_2=p\\q_1+q_2=q}}\; h_V(\sigma;p_1,q_1) \cdot h_W(\sigma;p_2,q_2)\, .
\]

In some situations, there is an alternative way to encode the multiplication type. We give two examples.

\begin{example}
\label{ssec:LPhi}
Let $\Phi \subset \Sigma$ be a CM type. To $(E,\Phi)$ we associate the polarizable Hodge structure $L_\Phi \in \QHS_{(E)}$ whose underlying vector space is~$E$ (viewed as $\mQ$-vector space with $E$-action), and for which $L_\Phi(\sigma)$ is purely of type $(-1,0)$ if $\sigma \in \Phi$ and purely of type $(0,-1)$ otherwise. The multiplication type of~$L_\Phi$ is just another way of encoding the CM type~$\Phi$. The Hodge structure~$L_\Phi$ determines a complex abelian variety~$A_\Phi$ up to isogeny with action by~$E$. 
\end{example}

\begin{example}
Let $V$ be a Hodge structure of K3 type with action by~$E$. Let $\tau$ be the unique complex embedding of~$E$ for which $V(\tau)^{1,-1} \neq 0$. We say that~$V$, with its given $E$-action, is of type $\KKK(E,\tau)$. The multiplication type is fully determined by~$\tau$ and $d = \dim_E(V)$; the only non-zero values are given by 
\[
h(\tau;p,q) = 
\begin{cases} 1 & \text{if $(p,q) = (1,-1)$}\\ d-1 & \text{if $(p,q) = (0,0)$}
\end{cases}
\qquad
h(\bar\tau;p,q) = 
\begin{cases} d-1 & \text{if $(p,q) = (0,0)$}\\ 1 & \text{if $(p,q) = (-1,1)$}
\end{cases}
\]
and $h(\sigma;0,0) = d$ if $\sigma \notin \{\tau,\bar\tau\}$.
\end{example}

\subsection{}
\label{ssec:AVEPhitau}
There is a simple way to relate Hodge structures of K3 type with $E$-action to abelian varieties. Before giving the precise statement, let us introduce some terminology.

As before, let $E$ be a CM field. Choose a CM type $\Phi\subset \Sigma$ and an embedding $\tau \in \Phi$. Let $H$ be a polarizable $\mQ$-Hodge structure with $E$-action which is purely of type $(-1,0) + (0,-1)$. In this case there is an abelian variety~$A$ with $E$-action such that $H \cong H_1(A,\mQ)$. We say that~$H$, with the given $E$-action, is of type $\AbV(E,\Phi,\tau)$ if, letting $d = \dim_E(H)$, we have
\[
h(\sigma;-1,0) = 
\begin{cases} 
d-1 & \text{if $\sigma = \tau$}\\ 
1 & \text{if $\sigma = \bar\tau$} \\ 
d & \text{if $\sigma \in \Phi\setminus\{\tau\}$} \\ 
0 & \text{if $\sigma \in \ol\Phi\setminus\{\bar\tau\}$}
\end{cases}
\]
Note that this uniquely determines~$\tau$ only if $d > 1$; this is the main case of interest for us. If $d=1$ then $H \cong L_\Psi$ where $\Psi = (\Phi \cup \{\bar\tau\})\setminus\{\tau\}$ is the CM type obtained from~$\Phi$ by replacing~$\tau$ with~$\bar\tau$.

\begin{proposition}
\label{prop:K3AVcorr}
Let $E$ be a CM field, $\tau$ a complex embedding of~$E$, and $\Phi$ a CM type for~$E$ with $\tau \in \Phi$. Then the functor given by $V \mapsto L_\Phi \otimes_E V$ gives an equivalence of categories
\[
\bigl(\text{\textup{Hodge structures of type $\KKK(E,\tau)$}}\bigr) 
\rightarrow 
\bigl(\text{\textup{Hodge structures of type $\AbV(E,\Phi,\tau)$}}\bigr)\, . 
\]
A quasi-inverse is given by $H \mapsto \ul\Hom_E(L_\Phi,H)$.
\end{proposition}

To avoid confusion, note that the two categories are taken to be the full subcategories of~$\QHS_{(E)}$ with the indicated collections of objects.

\begin{proof}
An easy calculation shows that the given functors land in the right category. That they are quasi-inverse to each other follows from the remark that $\ul\End_E(L_\Phi) \cong E(0)$, by which we mean $E$ with Hodge structure purely of type $(0,0)$.
\end{proof}

\begin{remark}
Let $B$ be a non-singular base variety. The correspondence in Prop.~\ref{prop:K3AVcorr} extends to variations of Hodge structure over~$B$, giving an equivalence $\cV \mapsto L_\Phi \otimes_E \cV$ between variations of Hodge structure over~$S$ of type $\KKK(E,\tau)$ and those of type $\AbV(E,\Phi,\tau)$. (The reader will have no trouble filling in the definitions of these notions.) Note that the ``twisting factor''~$L_\Phi$ is constant. In particular, the monodromy representations of~$\cV$ and $\cH = L_\Phi \otimes_E \cV$ are the same: if $V$ and $H = L_\Phi \otimes_E V$ are the fibres of~$\cV$ and~$\cH$ at some base point $b \in B$, and if $R_\cV \colon \pi_1(B,b) \to \GL_E(V)$ and $R_\cH\colon \pi_1(B,b) \to \GL_E(H)$ are the monodromy representations, then under the natural isomorphism $i \colon \GL_E(V) \isomarrow \GL_E(H)$ we have $i \circ R_\cV = R_\cH$.
\end{remark}

\section{Examples obtained from the Deligne-Mostow list}
\label{sec:ExaDMList}

\begin{notation}
\label{ssec:oldnew}
Let $H$ be a $\mQ$-Hodge structure equipped with an automorphism of order dividing~$m$. This makes $H$ a module over $\mQ\bigl[\mZ/m\mZ] = \prod_{d|m}\, \mQ(\zeta_d)$. Accordingly, we have a decomposition $H = \oplus_{d|m}\, H^{[d]}$, such that $H^{[d]}$ is a module over the cyclotomic field~$\mQ(\zeta_d)$. We call $H^\old = \oplus_{d<m}\, H^{[d]}$ the old part of~$H$ and $H^\new = H^{[m]}$ the new part. The same notation and terminology is applied to variations of Hodge structure.
\end{notation}

\subsection{}
\label{ssec:mNa}
Write $P = \mP^1$. In what follows we consider data $(m,N,a)$ consisting of integers $m \geq 2$ and $N \geq 3$, and an $N$-tuple $a = (a_1,\ldots,a_N)$ in $\{1,\ldots,m-1\}^N$, such that $\gcd(m,a_1,\ldots,a_N) = 1$ and $a_1 + \cdots + a_N \equiv 0 \bmod{m}$. Given such data, let $\tilde{M} = P^N\setminus \text{(diagonals)}$ be the variety of ordered $N$-tuples of distinct points on~$P$. There exists a smooth family of curves $\tilde{f}\colon \tilde{C} \to \tilde{M}$ whose fibre over $\lambda = (\lambda_1,\ldots,\lambda_N)$ with all $\lambda_i$ in~$\mA^1$ is the curve~$\tilde{C}_\lambda$ given by
\[
y^m = \prod_{i=1}^N\; (x-\lambda_i)^{a_i}\, .
\]
Let $\tilde{\alpha}$ be the automorphism of~$\tilde{C}$ over~$\tilde{M}$ given by $(x,y) \mapsto (x,\zeta_m \cdot y)$, with $\zeta_m = \exp(2\pi i/m)$.

The group $\Aut(P) \cong \PGL_2(\mC)$ acts freely on~$\tilde{M}$. This action lifts to an action on~$\tilde{C}$ that commutes with~$\tilde{\alpha}$. Dividing out by~$\Aut(P)$ we obtain a smooth family of curves
\[
f \colon C \to M
\]
and an automorphism $\alpha \in \Aut(C/M)$ of order~$m$. We have $\dim(M) = N-3$ and if $g$ is the genus of the curves~$C_\lambda$, the classifying morphism $M \to \cM_g$ is quasi-finite. 

If $d|m$ then the curve~$C_{d,\lambda}$ given by $y^d = \prod (x-\lambda_i)^{a_i}$ is  a quotient of~$C_\lambda$. The new part~$\cH^\new$ of the variation of Hodge structure $\cH = R^1f_*\mQ$ corresponds with an isogeny factor~$J^\new$ of the relative Jacobian~$J$, whose fibres $J_\lambda^\new$ are obtained as quotient of $J_\lambda = J(C_\lambda)$ by the abelian subvariety generated by the images of the $J(C_{d,\lambda}) \to J_\lambda$ for $d$ a proper divisor of~$m$.

\subsection{}
\label{ssec:CurveInvars}
The genus $g = \dim(J_\lambda)$ of the curves~$C_\lambda$ and the dimension of the new part $g^\new = \dim(J_\lambda^\new)$ are given by
\[
g = 1 + \frac{(N-2)m - \sum_{i=1}^N\, \gcd(m,a_i)}{2}
\qquad\text{and}\qquad
g^\new = \frac{(N-2)\cdot \varphi(m)}{2}\, .
\]
For $k \geq 1$, define $m_C^k(j)$ to be the multiplicity of $\zeta_m^j$ as an eigenvalue of~$\alpha^*$ acting on $H^0(C_\lambda,\omega^{\otimes k})$. The Chevalley-Weil formula gives
\begin{equation}
\label{eq:CW}
m_C^k(j) = \delta + \sum_{i=1}^N \left\langle \frac{(k-1) \cdot \gcd(a_i,m) - j\cdot a_i}{m} \right\rangle - \left(\frac{(k-1) \cdot \gcd(a_i,m)}{m} \right)\, ,
\end{equation}
where 
\[
\delta = (1-N) + (N-2)k + 
\begin{cases}
1 & \text{if $k=1$ and $j \equiv 0 \bmod{m}$,}\\
0 & \text{otherwise.}
\end{cases}
\]
(Recall from~\ref{ssec:notat} that~$\langle x\rangle$ denotes the fractional part of a number~$x$.) We write $m_C(j)$ for~$m^1_C(j)$; note that for $k=1$, \eqref{eq:CW} simplifies to
\begin{equation}
\label{eq:mC(j)}
m_C(j) = -1 + \sum_{i=1}^N\, \left\langle\frac{-ja_i}{m}\right\rangle \quad + \begin{cases}
1 & \text{if $j \equiv 0 \bmod{m}$,}\\
0 & \text{otherwise.}
\end{cases}
\end{equation}

\subsection{}
\label{ssec:mNaBall}
As input for the examples that we want to construct, we need data $(m,N,a)$ as in~\ref{ssec:mNa} such that the period domain for the variation~$\cH^\new$ is a complex $(N-3)$-ball. In the examples we will consider, $N \geq 4$ and $m\geq 3$, so that $E = \mQ(\zeta_m)$ is a CM field. With notation as in~\ref{ssec:mNa}, choose $\lambda \in M$ and let $H = H^1(C_\lambda,\mQ)$ be the fibre of~$\cH$ at~$\lambda$. If $\phi \colon H \otimes H \to \mQ(-1)$ denotes the polarization form, there is a unique skew-hermitian form~$\psi$ on the $E$-vector space~$H^\new$ such that $\trace_{E/\mQ} \circ \psi \colon H^\new \times H^\new \to \mQ$ equals the restriction of $(2\pi i) \cdot \phi$ to~$H^\new$. The Mumford-Tate group of~$H^\new$ is contained in the unitary group $\UU_E(H^\new,\psi)$. (See~\ref{ssec:notat} for the notation.)

We are interested in those cases where, abbreviating $\UU = \UU_E(H^\new,\psi)$, the associated real group is of the form
\[
\UU_\mR \cong \UU(1,N-3) \times \text{(compact factors)}\, .
\]
This happens if and only if there is a $j_0 \in (\mZ/m\mZ)^*$ such that $m_C(j_0) = 1$ and $m_C(j) \in \{0,N-2\}$ for all $j \neq \pm j_0$ in $(\mZ/m\mZ)^*$. Note that $m_C(j) + m_C(-j) = N-2$ for all $j \in (\mZ/m\mZ)^*$, in accordance with the fact that $\dim_E(H^\new) = N-2$.

Up to isomorphism, the family of curves $f\colon C \to M$ that we obtain from data $(m,N,a)$ does not change if we permute the~$a_i$ or if we change $a = (a_1,\ldots,a_N)$ by an element of $(\mZ/m\mZ)^*$. If the above condition is satisfied we may therefore assume that $j_0 = m-1$, which means that $\sum a_i = 2m$; additionally we may assume that $a = (a_1,\ldots,a_N)$ with $1 \leq a_1 \leq \cdots \leq a_N < m$. In this case, 
\begin{align}
\label{eq:mCBall}
\begin{split}
&m_C(-1) = 1\, ,\\
&m_C(1) = N-3\, ,\\
&m_C(j) \in \{0,N-2\}\ \text{for all $j \neq \pm 1$ in $(\mZ/m\mZ)^*$.}
\end{split}
\end{align}

Assuming now we are in this situation, let $Y \subset \Hom(\mS,\UU_\mR)$ be the $\UU(\mR)$-conjugacy class of the homomorphism $h \colon \mS \to \UU_\mR$ that defines the Hodge structure on~$H^\new$. Then $Y$ is a hermitian symmetric domain of non-compact type, and as such it is isomorphic to the complex $(N-3)$-ball. The pair $(\UU,Y)$ is a Shimura datum.

With $H^\new_\mZ = H^1(C_\lambda,\mZ) \cap H^\new$, let $\mL \subset \UU(\Af)$ be the compact open subgroup of elements that stabilize the lattice $H^\new_\mZ \otimes \Zhat$. Let $\Sh_\mL\bigl(\UU,Y\bigr)$ be the associated complex Shimura variety. (See Remark~\ref{rem:ShKStack}.) The tautological representation $\UU \to \GL(H^\new)$ gives rise to a variation of Hodge structure~$\cW$ over~$\Sh_\mL\bigl(\UU,Y\bigr)$. There is a well-defined morphism 
\begin{equation}
\label{eq:morpheta}
\eta \colon M \to \Sh_\mL\bigl(\UU,Y\bigr)
\end{equation}
such that $\cH^\new \cong \eta^* \cW$.   

The following is an immediate consequence of \cite{DelMost}, Prop.~3.9.

\begin{proposition}
\label{prop:DM3.9}
The morphism \eqref{eq:morpheta} is quasi-finite and dominant onto an irreducible component of $\Sh_\mL\bigl(\UU,Y\bigr)$.
\end{proposition}

Let us now make the connection with what was explained in Section~\ref{sec:Twisting}. Note that if we apply the construction of~\ref{ssec:mNa} to data $(m,3,b)$, this gives rise to a single curve~$D$. Concretely, it is the curve given by
\[
u^m = t^{b_1} (t-1)^{b_2}
\]
with automorphism~$\beta$ given by $(t,u) \mapsto (t,\zeta_m \cdot u)$.

\begin{proposition}
\label{prop:mNabMatch}
Let $(m,N,a)$ be data as in\/~\textup{\ref{ssec:mNa}} with $N \geq 4$, such that \eqref{eq:mCBall} holds. Let $f\colon C \to M$ be the associated family of curves, on which we have an automorphism~$\alpha$ of order~$m$. 
\begin{enumerate}[label=\textup{(\roman*)}]
\item Let $D$ be a curve with an automorphism~$\beta$ of order~$m$, such that the second cohomology of the surfaces
\begin{equation}
\label{eq:TlambdaDef}
T_\lambda = (C_\lambda \times D)/ \bigl\langle(\alpha,\beta)\bigr\rangle
\end{equation}
is of K3 type. Then $(D,\beta)$ is obtained, by the construction of\/~\textup{\ref{ssec:mNa}}, from data $(m,3,b)$, where $b = (b_1,b_2,b_3)$ is a triple with $1 \leq b_1 \leq b_2 \leq b_3 < m$ such that $\gcd(m,b_1,b_2,b_3) = 1$, and such that
\begin{align}
\label{abmatch}
\begin{split}
&b_1 + b_2 + b_3 = m\, , \text{and}\\
&\text{if $j \not\equiv \pm 1 \bmod{m}$ and $m_C(j) > 0$ then $\Bigl\langle \frac{j\cdot b_1}{m}\Bigr\rangle + \Bigl\langle \frac{j\cdot b_2}{m}\Bigr\rangle + \Bigl\langle \frac{j\cdot b_3}{m}\Bigr\rangle = 1$.}
\end{split}
\end{align}

\item Conversely, let $b = (b_1,b_2,b_3)$ be a triple satisfying the above conditions, let $(D,\beta)$ be the associated curve with an automorphism of order~$m$, and define~$T_\lambda$ by~\eqref{eq:TlambdaDef}. Then $H^2(T_\lambda,\mQ)\bigl(1\bigr)$ is of K3 type, and we have $H^2(T_\lambda,\mQ)\bigl(1\bigr) \cong V_0 \oplus V_1$, where $V_0$ is purely of type~$(0,0)$ and 
\[
V_1 \cong \bigl[H^{1,\new}(D,\mQ) \otimes_{\mQ(\zeta_m)} H^{1,\new}(C_\lambda,\mQ)\bigr](1) \cong \ul\Hom_{\mQ(\zeta_m)}\bigl(H_1^\new(D,\mQ),H_1^\new(C_\lambda,\mQ)\bigr)\, .
\]
\end{enumerate}
\end{proposition}

\begin{proof}
(\romannumeral1) Let $G \subset \Aut(C_\lambda \times D)$ denote the subgroup generated by the automorphism $(\alpha,\beta)$ of order~$m$. We have 
\[
H^2(T_\lambda,\mQ) \cong H^2(C_\lambda\times D,\mQ)^G \cong \mQ(-1)^{\oplus 2} \oplus \bigl[H^1(C_\lambda,\mQ) \otimes_\mQ H^1(D,\mQ)\bigr]^G\, .
\]
Similar to the notation we used for the curves~$C_\lambda$, let $m_D(j)$ denote the multiplicity of~$\zeta_m^j$ as an eigenvalue of~$\beta^*$ on $H^0(D,\omega_D)$. The condition that $H^2(T_\lambda,\mQ)\bigl(1\bigr)$ be of K3 type means that there is a unique $j_0 \in \mZ/m\mZ$ for which $m_C(j_0) \cdot m_D(-j_0) = 1$, and that $m_C(j) \cdot m_D(-j) = 0$ for all other~$j$.

If $H^2(T_\lambda,\mQ)\bigl(1\bigr)$ is of K3 type then necessarily $D/\langle\beta\rangle \cong \mP^1$. Indeed, if $g\bigl(D/\langle\beta\rangle\bigr) > 0$ then the Chevalley-Weil formula gives that $m_D(j) > 0$ for all $j \in (\mZ/m\mZ)^*$. But also, our assumption that $N \geq 4$ implies that $m_C(\pm 1) > 0$; so we arrive at a contradiction with the assumption that $h^{2,0}(T_\lambda) = 1$. The curve~$D$ is therefore given by data $(m,N^\prime,b)$ as in~\ref{ssec:mNa}. In this case, $m_D(-1) + m_D(1) = (N^\prime-2)$. Further, one of the numbers $m_C(-1) \cdot m_D(1) = m_D(1)$ and $m_C(1) \cdot m_D(-1) = (N-3) \cdot m_D(-1)$ equals~$1$ and the other equals~$0$. This implies that $N^\prime = 3$; further, if $N \geq 5$ we must have $m_D(1) = 1$ and $m_D(-1) = 0$, whereas for $N=4$ also $m_D(1) = 0$ and $m_D(-1) = 1$ is possible. As we can still change~$b$ by an element of $(\mZ/m\mZ)^*$, we may in either case assume that $m_D(1) = 1$ and $m_D(-1) = 0$, which is equivalent to the first condition in~\eqref{abmatch}. As now $m_C(-1) \cdot m_D(1) = 1$, there can be no $j \not\equiv -1 \bmod{m}$ for which $m_C(j) \cdot m_D(-j) > 0$; this gives the second condition in~\eqref{abmatch}. 

(\romannumeral2) If a triple~$b$ satisfies~\eqref{abmatch} then for the curve~$D$ associated with the data $(m,3,b)$ we have 
\[
m_C(j) \cdot m_D(-j) = \begin{cases} 1 & \text{if $j \equiv -1 \bmod{m}$}\\ 0 & \text{otherwise} \end{cases}
\]
and therefore $h^{2,0}(T_\lambda) =1$, i.e., $H^2(T_\lambda,\mQ)\bigl(1\bigr)$ is of K3 type. We have $H^2(T_\lambda,\mQ)\bigl(1\bigr) \cong V_0 \oplus V_1$ with
\[
V_0 = \mQ(0)^{\oplus 2} \oplus \bigl[H^{1,\old}(D,\mQ) \otimes_\mQ H^{1,\old}(C_\lambda,\mQ)\bigr]^G(1)\, ,
\]
which by construction is purely of type $(0,0)$, and
\[
V_1 \cong  \bigl[H^{1,\new}(D,\mQ) \otimes_\mQ H^{1,\new}(C_\lambda,\mQ)\bigr]^G(1)\cong
\bigl[H^{1,\new}(D,\mQ) \otimes_{\mQ(\zeta_m)} H^\new(C_\lambda,\mQ)\bigr](1)\, .
\]
Using duality, this last term can be rewritten as $\ul\Hom_{\mQ(\zeta_m)}\bigl(H_1^\new(D,\mQ),H_1^\new(C_\lambda,\mQ)\bigr)$.
\end{proof}

\subsection{}
\label{ssec:mNab}
Summing up, the examples we are interested in are obtained from data $(m,N,a,b)$, where:
\begin{enumerate}
\item  $(m,N,a)$ is a triple as in~\ref{ssec:mNa} with $m \geq 3$ and $N \geq 4$ such that $1 \leq a_1 \leq \cdots \leq a_N < m$ and such that \eqref{eq:mCBall} is satisfied; in particular, $m_C(-1) = 1$ means that $\sum a_i = 2m$;
\item $b = (b_1,b_2,b_3)$ with $1 \leq b_1 \leq b_2 \leq b_3 < m$ such that $\gcd(m,b_1,b_2,b_3) = 1$ and such that \eqref{abmatch} is satisfied.
\end{enumerate}
Given such data we have a family of curves $f \colon C \to M$ with an automorphism~$\alpha$ of order~$m$, and a curve~$D$ with an automorphism~$\beta$ of order~$m$. We refer to~$D$ as the ``twisting factor''; as the proof of the next result shows, $H_1(D,\mQ)$, with its action of $E = \mQ(\zeta_m)$, plays the role of~$L_\Phi$ in Proposition~\ref{prop:K3AVcorr}.

\begin{theorem}
\label{thm:TsGiveSpecial}
Let $(m,N,a,b)$ be data as in\/~\textup{\ref{ssec:mNab}}. For $\lambda \in M$, let $S_\lambda \to T_\lambda$ be a minimal resolution of singularities, and consider the family of surfaces $F \colon S \to M$ thus obtained. Choosing a base point in~$M$ and a polarization form~$\phi$ for the variation of Hodge structure $R^2F_*\mQ_S(1)$, let $\theta \colon M \to \Sh_\mK\bigl(\SO(\phi),X\bigr)$ be the associated morphism as in\/~\textup{\ref{ssec:famsurfS/B}}. Then $\theta$ is quasi-finite and the Zariski closure of its image is an $(N-3)$-dimensional special subvariety of $\Sh_\mK\bigl(\SO(\phi),X\bigr)$. By consequence, the set of $\lambda \in M$ for which $H^2(S_\lambda,\mQ)\bigl(1\bigr)$ is of CM type is analytically dense in~$M$.
\end{theorem}

\begin{proof}
Let $E = \mQ(\zeta_m)$, and consider the curve~$D$ as above. Let $\chi$ be the $E$-valued skew-hermitian form on $H_1^\new(D,\mQ)$ such that $\trace_{E/\mQ} \circ \chi$ is the natural polarization form. With notation as in~\ref{ssec:LPhi}, there is a unique CM type~$\Phi$ such that $H_1^\new(D,\mQ) \cong L_\Phi$  in~$\QHS_{(E)}$. To simplify notation we take this isomorphism as an identification. 

Choose a base point $\lambda \in M$, and write $H^\new = H^{1,\new}(C_\lambda,\mQ)$ as in \ref{ssec:mNaBall}. With $V_1$ as in~\ref{prop:mNabMatch}(\romannumeral2), we have $\ul\Hom_E(L_\Phi,H^\new) \isomarrow V_1$. We equip $V_1$ with the polarization form $\phi_1 = \trace_{E/\mQ} \circ (\chi^\vee \otimes \psi)$, and take as polarization on $V = H^2(T_\lambda,\mQ)\bigl(1\bigr) = V_0 \oplus V_1$ the orthogonal sum of~$\phi_1$ and a polarization form on~$V_0$. The composition
\[
\UU = \UU_E(H^\new,\psi) \xrightarrow{~\ul\Hom_E(L_\Phi,-)~} \SO(V_1,\phi_1) \longhookrightarrow \SO(V,\phi)
\]
is part of a morphism of Shimura data $u\colon (\UU,Y) \to \bigl(\SO(\phi),X\bigr)$ such that $\Sh(u) \circ \eta = \theta$ as morphisms $M \to \Sh_\mK\bigl(\SO(\phi),X\bigr)$. (Here $\eta$ is the morphism \eqref{eq:morpheta}.) The theorem now follows from Proposition~\ref{prop:DM3.9}.
\end{proof}

\subsection{}
Table~\ref{table:mNab} in Appendix~\ref{sec:appendix} contains a list of four-tuples $(m,N,a,b)$ that satisfy the conditions in~\ref{ssec:mNab}. We shall discuss there the relation with the work of Deligne and Mostow in \cite{DelMost} and~\cite{Mostow}. All these examples give rise to families $S \to M$ of surfaces that, by the theorem just proved, provide a positive answer to Question~\ref{question:B}.
  
It remains to analyze how these examples fit into the classification of surfaces. The next theorem gives the result. The details of the proof are given in the next section.

\begin{theorem}
\label{thm:MainThm}
Let $(m,N,a,b)$ be data as listed in\/~\textup{Table~\ref{table:mNab}}, and let $S \to M$ be the associated family of surfaces. For $\lambda \in M$ let $S_\lambda^{\min}$ be the minimal model of~$S_\lambda$. Note that, by construction, $q(S_\lambda^{\min}) = 0$ and $p_\geom(S_\lambda^{\min}) = 1$.

\begin{enumerate}[label=\textup{(\roman*)}]
\item In the examples numbered \textup{42}, \textup{46}, \textup{95} and~\textup{122} the surfaces $S_\lambda^{\min}$ are of general type with $K^2 = 1$.
\item In example \textup{59} the surfaces $S_\lambda^{\min}$ are of general type with $K^2 = 2$.
\item In example~\textup{60} the surfaces $S_\lambda^{\min}$ are of general type with $K^2 = 3$.
\item In all other examples the surfaces $S_\lambda^{\min}$ are K3 surfaces. 
\end{enumerate}
\end{theorem}

\section{Analyzing the examples}

The only difficulty in proving Theorem~\ref{thm:MainThm} comes from the fact that the surfaces~$S_\lambda$ are (usually) not minimal, and that we have to find which divisors can be blown down. In any given example this can be analyzed by hand; as this is rather cumbersome we use a computer program to do some of the work.

We start by showing what happens in two concrete examples; this may help the reader to understand what is going on. After that we give some general facts about the surfaces~$S_\lambda$, and we explain the proof of Thm.~\ref{thm:MainThm}.

\begin{notation}
\label{ssec:KXnotat}
If $X$ is a non-singular surface with $p_\geom(X) = 1$, we denote by~$\cK_X$ the unique canonical effective divisor on~$X$.
\end{notation}

\begin{example}
\label{exa:ElliptK3}
We start with a fairly simple example. We will not give full details for all calculations, as these are special cases of the general results discussed later in this section.

Take $N=8$ and $m=4$ with $a=(1,1,1,1,1,1,1,1)$ and $b = (1,1,2)$. This is example~138 in Table~\ref{table:mNab}. The genus~$9$ curves~$C_\lambda$ are given by $y^4 = \prod_{i=1}^8\, (x-\lambda_i)$, with automorphism~$\alpha$ of order~$4$ given by $(x,y) \mapsto (x,\zeta_4 \cdot y)$. The twisting factor~$D$ is the elliptic curve with $j=1728$ given by $u^4 = t(t-1)$, with automorphism~$\beta$ given by $(t,u) \mapsto (t,\zeta_4\cdot u)$. 

Consider the morphism $x \colon C_\lambda \to \mP^1$. Above $\lambda_i$ we have a unique point $P_i \in C_\lambda$. Similarly, let $\mu_1 = 0$, $\mu_2 = 1$ and $\mu_3 = \infty$, and consider the morphism $t \colon D \to \mP^1$. Above~$\mu_j$ there is a unique point~$Q_j \in D$ if $j=1,2$, while above~$\mu_3$ there are two points~$Q_3^{(\ell)}$, $\ell = 1,2$.

In the rest of the example we fix $\lambda = (\lambda_1,\ldots,\lambda_8)$. Let $T = (C_\lambda \times D)/G$, where $G \subset \Aut(C_\lambda \times D)$ is the subgroup generated by $(\alpha,\beta)$. For $i \in \{1,\ldots,8\}$, let $\bar{Y}_i \subset T$ be the image of $\{P_i\} \times D$; these are all isomorphic to~$\mP^1$. Similarly, let $\bar{Z}_j$ be the image of $C_\lambda \times \{Q_j\}$. Then $\bar{Z}_1 \cong \bar{Z}_2 \cong \mP^1$, while $\bar{Z}_3$ is a curve of genus~$3$. The components $\bar{Y}_i$ and~$\bar{Z}_j$ intersect transversally in a single point~$R_{ij}$. For $j \in \{1,2\}$ these are quotient singularities of type $\frac{1}{4}(1,1)$, whereas for $j=3$ the~$R_{ij}$ are ordinary double points.

Let $S \to T$ be a minimal resolution of singularities, and let $Y_i$, $Z_j \subset S$ be the strict transforms of~$\bar{Y}_i$ and~$\bar{Z}_j$. For each pair $(i,j)$ there is an irreducible exceptional divisor~$E_{ij}$ connecting $Y_i$ and~$Z_j$, and we have $E_{ij}^2 = -4$ for $j = 1,2$, whereas $E_{i3}^2 = -2$. See the figure.

\begin{figure}
\begin{center}
\begin{tikzpicture}[x=.3cm,y=.3cm,scale=0.8]
\draw (.5,5)--(46,5)--(46,24.5)--(.5,24.5)--(.5,5);
\foreach \j in {1,2,3}
	\draw[line width=1pt] (2.5,5+5*\j)--(43.5,5+5*\j);
\foreach \i in {1,...,8}
	\draw[-,line width=8pt,draw=white] (5*\i,7)--(5*\i,23);
\foreach \i in {1,...,8}
	\draw[line width=1pt] (5*\i,7)--(5*\i,23);
\foreach \i in {1,...,8}
	\foreach \j in {1,2,3}
		\draw[line width=1pt] (-.5+5*\i,2.5+5*\j)--(2.5+5*\i,5.5+5*\j);
\foreach \i in {1,...,8}
	\node at (5*\i,23.5) {$\scriptstyle -1$};
\foreach \i in {1,...,8}	
	\node at (5*\i,6) {$\scriptstyle Y_{\i}$};	
\foreach \j in {1,2,3}
	\node at (44.5,5+5*\j) {$\scriptstyle Z_{\j}$};
\foreach \i in {1,...,8}
	\foreach \j in {1,2}
		\node at (1.5+5*\i,3.5+5*\j) {$\scriptstyle -4$};
\foreach \i in {1,...,8}
	\node at (1.5+5*\i,18.5) {$\scriptstyle -2$}; 
\node at (1.5,10) {$\scriptstyle -2$};
\node at (1.5,15) {$\scriptstyle -2$};
\node at (1.5,20) {$\scriptstyle -4$};
\node at (-.5,22) {$S$};
\end{tikzpicture}
\end{center}
\end{figure}

By \cite{MistPol}, Cor.~3.6, $K_S^2 =  -16$. The eight components~$Y_i$ are $(-1)$-curves. After blowing them down, the~$E_{i3}$ become $(-1)$-curves, and after blowing these down we arrive at a minimal surface~$S^{\min}$ with $K^2= 0$. 
 
The projection $S \to T \to C_\lambda/G \cong \mP^1$ factors via~$S^{\min}$ and exhibits~$S^{\min}$ as an isotrivial elliptic surface over~$\mP^1$. The singular fibres are the fibres above the eight points~$\lambda_i$; they are all of type~III. The inclusion $C_\lambda \hookrightarrow C_\lambda \times D$ given by $P \mapsto (P,Q_1)$ is a section of the projection $C_\lambda \times D \to C_\lambda$. Dividing out the action of~$G$, we obtain a section of $S^{\min} \to \mP^1$.

By \cite{FriedMorg}, Section~2.2.1, Prop.~2.1, $S^{\min}$ is simply connected. As $p_\geom(S^{\min}) = 1$, it follows from ibid., section~1.3.6, Propositions 3.22 and~3.23(\romannumeral1) that $S^{\min}$ is a K3 surface. In this way we find a family of K3 surfaces depending on $5$ moduli. It is the same as the family described in ibid., Section~1.4.2 (page~62), with $g = g(\mP^1) =0$ and $d=2$. 
\end{example}

\begin{example}
Take $m=14$ and $N=4$ with $a = (1,9,9,9)$ and $b = (3,4,7)$. This is example~60 in our table. The curves~$C_\lambda$ have genus~$13$, the twisting factor~$D$ is a curve of genus~$3$. 

We fix $\lambda = (\lambda_1,\ldots,\lambda_4)$. Consider the quotient morphism $x \colon C_\lambda \to C_\lambda/\langle\alpha\rangle \cong \mP^1$. Above~$\lambda_i$ there is a unique point $P_i \in C_\lambda$. Similarly, we have $t \colon D \to \mP^1$, and if we set $\mu_1 = 0$, $\mu_2 = 1$ and $\mu_3 = \infty$, there is a unique point~$Q_j \in D$ above each~$\mu_j$.

Let $q \colon (C_\lambda\times D) \to T = T_\lambda$ be the quotient morphism, and consider the curves $\bar{Y}_i = q(\{P_i\} \times D)$ ($i=1,\ldots,4$) and $\bar{Z}_j = q(C_\lambda \times \{Q_j\})$ ($j=1,2,3$). The $\bar{Y}_i$ are all rational, and so is~$\bar{Z}_1$. The curves~$\bar{Z}_2$ and~$\bar{Z}_3$ have $g(\bar{Z}_2) = 6$ and $g(\bar{Z}_3) = 1$. The $\bar{Y}_i$ and~$\bar{Z}_j$ intersect transversally in a single point~$R_{ij}$. The type of singularities we find at these points is given by the following table, in which an entry $(n,q)$ indicates that we have a cyclic quotient singularity of type $\frac{1}{n}(1,q)$.
\[
\begin{array}{|c|c|c|c|l|}
\hline
(2,1) & (2,1) & (2,1) & (2,1) & Z_3 \\
\hline
(7,4) & (7,1) & (7,1) & (7,1) & Z_2 \\
\hline
(14,5) & (14,3) & (14,3) & (14,3) & Z_1 \\
\hline
Y_1 & Y_2 & Y_3 & Y_4 & \\
\hline
\end{array}
\]
Resolving singularities we obtain a surface~$S$ that looks as follows.

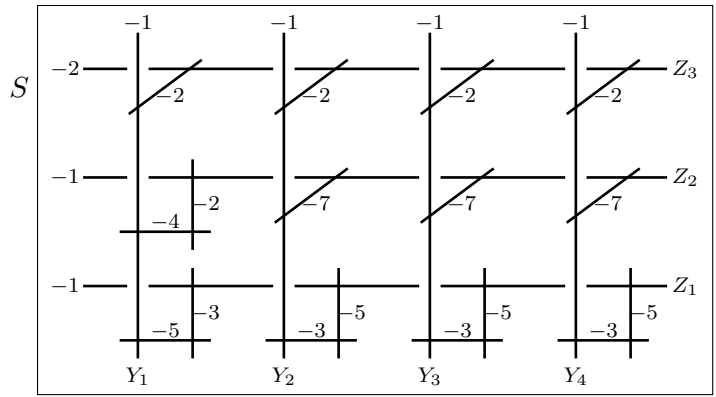
\begin{center}
\begin{tikzpicture}[x=.3cm,y=.3cm,scale=0.8]
\draw (2.5,5)--(39.5,5)--(39.5,26.5)--(2.5,26.5)--(2.5,5);
\foreach \j in {1,2,3}
	\draw[line width=1pt] (5,5+6*\j)--(37,5+6*\j);
\foreach \i in {1,...,4}
	\draw[-,line width=8pt,draw=white] (8*\i,7)--(8*\i,25);
\foreach \i in {1,...,4}
	\draw[line width=1pt] (8*\i,7)--(8*\i,25);
\foreach \i in {2,3,4}
	\foreach \j in {2,3}
		\draw[line width=1pt] (-.5+8*\i,2.5+6*\j)--(3.5+8*\i,5.5+6*\j);
\draw[line width=1pt] (7.5,20.5)--(11.5,23.5);
\foreach \i in {1,...,4}
	\draw[line width=1pt] (-1+8*\i,8)--(4+8*\i,8);
\foreach \i in {1,...,4}
	\draw[line width=1pt] (3+8*\i,7)--(3+8*\i,12);
\draw[line width=1pt] (7,14)--(12,14);
\draw[line width=1pt] (11,13)--(11,18);
\foreach \i in {1,...,4}
	\node at (8*\i,25.5) {$\scriptstyle -1$};
\foreach \i in {1,...,4}	
	\node at (8*\i,6) {$\scriptstyle Y_{\i}$};	
\foreach \j in {1,2,3}
	\node at (38,5+6*\j) {$\scriptstyle Z_{\j}$};
\foreach \i in {1,...,4}
	\node at (1.75+8*\i,21.5) {$\scriptstyle -2$};
\foreach \i in {2,3,4}
	\node at (1.75+8*\i,15.5) {$\scriptstyle -7$};
\foreach \i in {2,3,4}
	\node at (3.7+8*\i,9.5) {$\scriptstyle -5$};
\node at (11.7,9.5) {$\scriptstyle -3$};	
\node at (11.7,15.5) {$\scriptstyle -2$};
\foreach \i in {2,3,4}
	\node at (1.5+8*\i,8.5) {$\scriptstyle -3$};
\node at (9.5,8.5) {$\scriptstyle -5$};	
\node at (9.5,14.5) {$\scriptstyle -4$};
\node at (4,11) {$\scriptstyle -1$};
\node at (4,17) {$\scriptstyle -1$};
\node at (4,23) {$\scriptstyle -2$};
\node at (1.5,22) {$S$};
\end{tikzpicture}
\end{center}
Using \cite{MistPol}, Cor.~3.6, we find that $K_S^2 =  -9$. To obtain a minimal model, we first blow down the four components~$Y_i$. After that we blow down the images of the four $(-2)$-curves connecting $Y_i$ and~$Z_3$. The three $(-3)$-components in the HJ-strings connecting $Y_i$ ($i=2,3,4$) and~$Z_1$ have now become $(-1)$-curves, and we blow down these. Finally we blow down the image of~$Z_1$. At that point we have arrived at a minimal model~$S^{\min}$ with $K^2=3$. (Note that $Z_2$ is not a rational curve.)
\end{example}

\subsection{}
Let us now turn to the general case. Throughout, $(m,N,a,b)$ are data as in~\ref{ssec:mNab} and if $n$ is an integer, we define $\zeta_n = \exp(2\pi i/n)$. Let $f \colon C \to M$ be the family of curves associated with $(m,N,a)$, as in~\ref{ssec:mNa}. Concretely, if $\lambda \in M$ is the class of an $N$-tuple $(\lambda_1,\ldots,\lambda_N)$, the fibre~$C_\lambda$ is the curve given by $y^m = \prod_{i=1}^N\; (x-\lambda_i)^{a_i}$, with automorphism~$\alpha$ of order~$m$ given by $(x,y) \mapsto (x,\zeta_m \cdot y)$.

Consider the morphism $x \colon C_\lambda \to \mP^1$. Over $\lambda_i$ there are $\gcd(a_i,m)$ points~$P_i^{(k)}$ in~$C_\lambda$, all with ramification index $m/\gcd(a_i,m)$. The stabilizer of~$P_i^{(k)}$ in~$\langle \alpha\rangle$ is the subgroup generated by~$\alpha^{\gcd(a_i,m)}$. The latter automorphism acts on the tangent space at~$P_i^{(k)}$ as multiplication by~$\zeta_{m/\gcd(a_i,m)}^{c(i)}$, where $c(i)$ is the inverse of $a_i/\gcd(a_i,m)$ in $(\mZ/\frac{m}{\gcd(a_i,m)}\mZ)^*$.

We have normalized the $N$-tuple $a = (a_1,\ldots,a_N)$ in such a way that $\sum_{i=1}^N\, a_i = 2m$; by the Chevalley-Weil formula~\eqref{eq:CW} this implies that there is a $1$-dimensional space of regular $1$-forms on~$C_\lambda$ on which $\alpha^*$ is multiplication by~$\zeta_m^{-1}$. In fact, direct calculation gives that this space is spanned by $\mathrm{d}x/y$, and we have
\begin{equation}
\label{eq:divdx/y}
\div\left(\frac{\mathrm{d}x}{y}\right) = \sum_{i=1}^N \sum_{k=1}^{\gcd(a_i,m)}\; \left(\frac{m-a_i}{\gcd(a_i,m)} - 1\right) \cdot P_i^{(k)}\, .
\end{equation}

The ``twisting factor'' is the curve~$D$ obtained from the triple~$(m,3,b)$; concretely it is the curve given by $u^m = t^{b_1}(t-1)^{b_2}$, with automorphism~$\beta$ of order~$m$ given by $(t,u) \mapsto (t,\zeta_m \cdot u)$. Let $\mu_1 = 0$, $\mu_2 = 1$ and~$\mu_3 = \infty$, and consider the morphism $t \colon D \to \mP^1$. Over $\mu_j$ there are $\gcd(b_j,m)$ points~$Q_j^{(\ell)}$ in~$D$, all with ramification index $m/\gcd(b_j,m)$. The stabilizer of~$Q_j^{(\ell)}$ in $\langle\beta\rangle \subset \Aut(D)$ is the subgroup generated by~$\beta^{\gcd(b_j,m)}$. The latter automorphism acts on the tangent space as multiplication by~$\zeta_{m/\gcd(b_j,m)}^{d(j)}$, where $d(j)$ is the inverse of $b_j/\gcd(b_j,m)$ in $(\mZ/\frac{m}{\gcd(b_j,m)}\mZ)^*$, . 

Again by Chevalley-Weil, the fact that $\sum_{j=1}^3\, b_j = m$ implies that there is a $1$-dimensional space of regular $1$-forms on~$D$ on which $\beta^*$ is multiplication by~$\zeta_m$; this space is spanned by $t^{b_1-1}(t-1)^{b_2-1}\mathrm{d}t/u^{m-1}$, and
\begin{equation}
\label{eq:divformD}
\div\left(\frac{t^{b_1-1}(t-1)^{b_2-1}\mathrm{d}t}{u^{m-1}}\right) = \sum_{j=1}^3 \sum_{\ell=1}^{\gcd(b_j,m)}\; \left(\frac{b_j}{\gcd(b_j,m)} - 1\right) \cdot Q_j^{(\ell)}\, .
\end{equation}

\subsection{}
\label{ssec:Tsing}
Define $T_\lambda = (C_\lambda \times D)/G$, where $G \cong \mZ/m\mZ$ is the subgroup of $\Aut(C_\lambda \times D)$ generated by $\gamma = (\alpha,\beta)$. In what follows we fix $\lambda \in M$ and abbreviate~$T_\lambda$ to~$T$. Let $q\colon C_\lambda \times D \to T$ be the quotient morphism, and let 
\[
C_\lambda/\langle\alpha\rangle \cong \mP^1 \xleftarrow{~\pr_1~} T \xrightarrow{~\pr_2~} \mP^1 \cong D/\langle\beta\rangle 
\]
be the two projections. For $i \in \{1,\ldots,N\}$, let $\bar{Y}_i \subset T$ be the divisor obtained as the image of $\{P_i^{(k)}\} \times D$ (for any~$k$), which is the reduced fibre of~$\pr_1$ over~$\lambda_i$. Similarly, for $j \in \{1,2,3\}$, let $\bar{Z}_j \subset T$ be the image of $C_\lambda \times \{Q_j^{(\ell)}\}$ (for any~$\ell$), which is the reduced fibre of~$\pr_2$ over~$\mu_j$.

Fix $i \in \{1,\ldots,N\}$ and $j \in \{1,2,3\}$. The curves~$\bar{Y}_i$ and~$\bar{Z}_j$ intersect in $\gcd(a_i,b_j,m)$ distinct points; these correspond to the $G$-orbits in the set of points $(P_i^{(k)},Q_j^{(\ell)})$, for $1\leq k \leq \gcd(a_i,m)$ and $1 \leq \ell \leq \gcd(b_j,m)$. The stabilizer in $G = \langle \gamma\rangle$ of any of these points is the subgroup generated by~$\gamma^h$, where $h = \lcm\bigl(\gcd(a_i,m),\gcd(b_j,m)\bigr)$. Let $c(i)$ and~$d(j)$ be as above and write $n = m/h$; then $\gamma^h$ acts on the tangent space by $(X,Y) \mapsto (\zeta_n^{c(i)} \cdot X,\zeta_n^{d(j)} \cdot Y)$. If $q = d(j) \cdot c(i)^{-1} \in \{1,\ldots,n-1\}$ is the representative of $(a_i/\gcd(a_i,m)) \cdot (b_j/\gcd(b_j,m))^{-1}$ in $(\mZ/n\mZ)^*$, the conclusion is that $\bar{Y}_i \cap \bar{Z}_j$ consists of $\gcd(a_i,b_j,m)$ points, at each of which $T$ has a cyclic quotient singularity given by the pair of numbers $(n,q)$.

\begin{remark}
If $q^\prime \in \{1,\ldots,n-1\}$ represents the inverse of~$q$ modulo~$n$, the pairs $(n,q)$ and $(n,q^\prime)$ give isomorphic singularities. However, if we say that $T$ has a singularity of type $(n,q)$ at $R = q(P,Q)$ we mean that there is a primitive $n$th root of unity~$\xi$ such that, with notation as above, $\gamma^h$ acts on the tangent space $T_P(C_\lambda) \oplus T_Q(D)$ by $(X,Y) \mapsto (\xi \cdot X,\xi^q \cdot Y)$. This normalization plays a role in the calculations that follow.
\end{remark}

\subsection{}
\label{ssec:SelfInt}
Let $r \colon S \to T$ be a minimal resolution of singularities. As before, if there is no risk of confusion we fix~$\lambda$ and write~$S$ instead of~$S_\lambda$. Let $Y_i, Z_j \subset S$ be the strict transforms of $\bar{Y}_i$ and~$\bar{Z}_j$. 

Let $i \in \{1,\ldots,N\}$ and $j \in \{1,2,3\}$. Let $n$ and~$q$ be calculated as in~\ref{ssec:Tsing}, and let $n/q = [b_1,\ldots,b_t]$ be the Hirzebruch-Jung continued fraction expansion. (See \cite{MistPol}, Section~2.) For $R \in \bar{Y}_i \cap \bar{Z}_j$ the exceptional fibre~$r^{-1}\{R\}$ consists of a chain $E_1,\ldots,E_t$ of rational curves such that in the sequence $Z_j,E_1,\ldots,E_t,Y_i$ each component transversally intersects the next and $E_k^2 = -b_k$. The chain of components~$E_k$ is referred to as the HJ-string in the resolution. 

By \cite{Polizzi}, Prop.~2.8, if $\bar{Z}_j$ contains singular points of type $(n_1,q_1),\ldots,(n_r,q_r)$ then $(Z_j)^2 = - \sum_{\nu=1}^r\,  q_\nu/n_\nu$. The same result can be used to calculate the self-intersections~$(Y_i)^2$, but we must change the role of the two coordinates. The result is that if~$\bar{Y}_i$  contains singular points of type $(n_1,q_1),\ldots,(n_s,q_s)$ we have $(Y_i)^2 = - \sum_{\nu=1}^r\,  q^\prime_\nu/n_\nu$, where $q^\prime_\nu$ represents the multiplicative inverse of~$q_\nu$ modulo~$n_\nu$.

\subsection{}
Let $T^\circ \subset T$ denote the regular locus. As the dualizing sheaf~$\omega_T$ is reflexive,
\[
H^0(T,\omega_T) = H^0(T^\circ,\omega_T) \isomarrow H^0(C\times D_\lambda,\Omega^1_{C_\lambda} \otimes \Omega^1_D)^G\, .
\]
By construction, this space is $1$-dimensional, and it follows from \eqref{eq:divdx/y} and~\eqref{eq:divformD} that the unique effective Weil divisor~$\cK_T$ representing~$\omega_T$ is given by
\[
\cK_T = \sum_{i=1}^N \; \left(\frac{m-a_i}{\gcd(a_i,m)} - 1\right) \cdot \bar{Y}_i + \sum_{j=1}^3 \; \left(\frac{b_j}{\gcd(b_j,m)} - 1\right) \cdot \bar{Z}_j\, .
\]

Because the singularities on~$T$ are all rational, $\omega_T = r_*\omega_S$. If $\cE$ is the collection of all irreducible components $E \subset S$ that appear in an exceptional fibre of~$r$, it follows (with notation as in~\ref{ssec:KXnotat}) that
\[
\cK_S = \tilde\cK_T + \sum_{E \in \cE}\, n_E \cdot E\, ,
\]
where $\tilde\cK_T$ is the strict transform of~$\cK_T$ and $n_E \geq 0$ for all~$E$.

If $R \in T$ is a singular point of type $(n,q)$ and $n/q = [b_1,\ldots,b_t]$ is the continued fraction expansion, one defines an invariant~$h_R$ of the singularity by
\[
h_R = 2 - \frac{2+q+q^\prime}{2} - \sum_{k=1}^t (b_k-2)\, ,
\]
where $q^\prime \in \{1,\ldots,n-1\}$ represents the inverse of~$q$ modulo~$n$. By \cite{MistPol}, Cor.~3.6, the self-intersection~$K_S^2$ is given by
\begin{equation}
\label{eq:KS2}
K_S^2 = \frac{8(g_{C_\lambda}-1)(g_D-1)}{m} + \sum_{R \in \mathrm{Sing}(T)}\, h_R\, .  
\end{equation}
In almost all our examples, $K_S^2 < 0$ and $S$ is not minimal. If $S^{\min}$ is the minimal model of~$S$, the irreducible curves that are contracted under $S \to S^{\min}$ all have positive coefficient in~$\cK_S$. It follows that the only curves in~$S$ that may be contracted are
\begin{enumerate}
\item\label{item:whichYi} 
the rational components $Y_i$ for which $(m-a_i)/\gcd(a_i,m) > 1$;
\item\label{item:whichZj} 
the rational components $Z_j$ for which $b_j/\gcd(b_j,m) > 1$ (i.e., $b_j \nmid m$);
\item\label{item:EincE} 
the exceptional components $E \in \cE$.
\end{enumerate}
We can refine this, for if $E_1,\ldots,E_t$ is a HJ-string in the resolution that connects~$Z_j$ and~$Y_i$, and if neither~$Y_i$ nor~$Z_j$ is contracted to a point in~$S^{\min}$, the components~$E_i$ cannot be blown down. So in~\ref{item:EincE} we may restrict to the subcollection $\cE^\prime \subset \cE$ of components~$E$ appearing in a HJ-string for which either $Y_i$ is among the components in~\ref{item:whichYi} or $Z_j$ is among the components in~\ref{item:whichZj}.

\subsection{}
For the proof of Theorem~\ref{thm:MainThm}, the main point is to calculate $K^2$ on the minimal model~$S^{\min}$. For this we first calculate~$K_S^2$ using~\eqref{eq:KS2}. Then we draw up a list of all rational components $Y_i$, $Z_j$ that occur with positive coefficient in~$\cK_S$, as in \ref{item:whichYi} and~\ref{item:whichZj} above, together with all exceptional components $E \in \cE^\prime$. Let $A$ be the intersection matrix for this collection of components, using the results described in~\ref{ssec:SelfInt}.

Using a python script we then calculate how often we can blow down a component. For this, in each step of the process we simply search for the next $(-1)$-component; if there is one, we replace the intersection matrix by the (smaller) matrix we get after blowing down. Also, if $Z$ is the $(-1)$-component that we are blowing down, we remove all rows and columns corresponding to components~$Z^\prime$ with $Z \cdot Z^\prime \geq 2$, as these will become singular curves after blowing down.

The python script used is available from the author upon request. The values that are found for~$K^2$ on~$S^{\min}$ are listed in the last column of Table~\ref{table:mNab}. The proof of Theorem~\ref{thm:MainThm} is then completed by using the following result.

\begin{proposition}
With $S = S_\lambda$ for some $\lambda \in M$, let $S^{\min}$ be the minimal model of~$S$. If $S^{\min}$ has $K^2 = 0$ then it is a K3 surface. If $K^2 > 0$ then $S^{\min}$ is of general type.
\end{proposition}

\begin{proof}
To simplify notation, write $X = S^{\min}$. We need to show that $X$ cannot have Kodaira dimension~$1$. Suppose, to the contrary, that $\kod(X) = 1$. By \cite{Badescu}, Theorems~7.13 and~9.4, there exists an elliptic fibration $f \colon X \to B$. Because $q(X) = 0$ and $p_\geom(X) = 1$, we find that $B \cong \mP^1$ and $R^1f_* \cO_X \cong \cO_B(-2)$. In particular, there are no exceptional fibres. If $n_1H_1, \ldots, n_rH_r$ are the multiple fibres ($n_i \geq 2$ and $H_i \subset X$ a reduced fibre), ibid., Theorem~7.15(c) gives
\[
\omega_X \cong \cO_X\left(\sum_{i=1}^r\, (n_i-1) H_i\right)\, .
\]

There is a co-finite set $B^\circ \subset B$ such that for $b \in B^\circ$ the fibre $F_b = f^{-1}\{b\}$ is an non-singular curve of genus~$1$. Let $\hat{F}_b \subset S$ be its inverse image in~$S$. Possibly after shrinking~$B^\circ$ we have $\hat{F}_b \isomarrow F_b$ for all $b \in B^\circ$, and as the $\hat{F}_b$ form an algebraic family of mutually disjoint curves, $\hat{F}_b^2 = 0$. Hence $\hat{F}_b \cdot \cK_S = 2g(\hat{F}_b) - 2 = 0$. All irreducible components of~$\cK_S$ that are contracted under $S \to X = S^{\min}$ are rational curves. As $F_b$ is not contained in the support of~$\cK_X$, it follows that $\hat{F}_b$ is not contained in the support of~$\cK_S$, so $\hat{F}_b \cdot \cK_S = 0$ implies that $\hat{F}_b$ is even disjoint from~$\cK_S$. As there are always vertical and horizontal fibres~$Y_i$ and~$Z_j$ that occur with positive multiplicity in~$\cK_S$ this is possible only if the~$\hat{F}_b$ are contracted under both projections $\pr_i \colon S \to \mP^1$; but this clearly cannot be the case.
\end{proof}

\begin{remark}
If $\varphi(m) = 2$ (i.e., $m \in \{3,4,6\}$), the same analysis as in Example~\ref{exa:ElliptK3} shows that the surfaces~$S_\lambda$ are (isotrivial) elliptic K3 surfaces.
\end{remark}

\begin{remark}
In some cases there is a quick way to show that $S^{\min}$ is a K3 surface. The strategy for this is the following. Let $T^\circ \subset T$ be the regular locus, $S^\circ \subset S$ its inverse in~$S$, so that $r \colon S\to T$ restricts to an isomorphism $S^\circ \isomarrow T^\circ$. For $k > 0$,
\[
H^0(S,\omega_S^{\otimes k}) \longhookrightarrow H^0(S^\circ,\omega_{S^\circ}^{\otimes k}) = H^0(T^\circ,\omega_{T^\circ}^{\otimes k}) = H^0(C_\lambda \times D,\omega_{C_\lambda \times D}^{\otimes k})^G\, ,
\]
where we use that on $C_\lambda \times D$ we only have isolated points with non-trivial stabilizer. Analogous to the notation $m_C^{(k)}(j)$ introduced in~\ref{ssec:CurveInvars}, write $m_D^{(k)}(j)$ for the dimension of the $\zeta_m^j$-eigenspace of~$\beta^*$ in $H^0(D,\omega_D^{\otimes k})$. For the plurigenera of~$S$ this gives the estimate
\[
P_k(S) \leq h^0(C_\lambda \times D,\omega_{C_\lambda \times D}^{\otimes k})^G = \sum_{j\in (\mZ/m\mZ)}\, m_C^{(k)}(j) \cdot m_D^{(k)}(-j)\, .
\] 
On the other hand, it is easily seen from the Chevalley-Weil formula \eqref{eq:CW} that
\[
m_C^{(k+\ell m)}(j) = m_C^{(k)}(j) + \ell \cdot \delta_C\, ,\qquad
m_D^{(k+\ell m)}(j) = m_D^{(k)}(j) + \ell \cdot \delta_D\, , 
\]
where the constants $\delta_C$, $\delta_D$ are given by
\[
\delta_C = (N-2)m - \sum_{i=1}^N\, \gcd(a_i,m)\, ,\qquad
\delta_D = m - \sum_{j=1}^3\, \gcd(b_j,m)\, .
\]
Hence if we write the function $\ell \mapsto \sum_{j\in (\mZ/m\mZ)}\, m_C^{(k+\ell m)}(j) \cdot m_D^{(k+ \ell m)}(-j)$ as a polynomial in~$\ell$, the quadratic term has coefficient $m \cdot \delta_C \cdot \delta_D$.  

If $S$ is of general type, $P_k(S) = 2 + \binom{k}{2}\cdot K_{S^{\min}}^2$ and $K_{S^{\min}}^2 \geq 1$. In particular, we must have $\delta_C \cdot \delta_D \geq m/2$. In about one third of all examples in Table~\ref{table:mNab} this fails; in these cases we can immediately conclude that $S^{\min}$ is a K3 surface.
\end{remark}

\section{Further examples}
\label{sec:Further}

\subsection{}
If we drop the assumption that $p_\geom = 1$ there are many more examples we can make. For instance, one can start with a family of curves $C \to M$ for which the image of the period map $M \to A_g$ is a special subvariety of positive dimension. See for instance Tables~1 and~2 in~\cite{MO} for examples of such. Next take any curve~$D$ whose Jacobian is of CM type, for instance a curve obtained as a cyclic cover of~$\mP^1$ ramified over three points. The family of surfaces $C_\lambda \times D$ (or any quotient thereof) then contains infinitely many CM fibres.

\subsection{}
\label{ssec:furtherexa}
More interesting is that we vary on the construction given in the previous sections. In the examples discussed above, we have assumed that the new part of the Jacobians $J(C_\lambda)$ is parametrized by a ball quotient. We obtain some further examples by starting with a family of curves $C \to M$ given by data $(m,N,a)$ for which the new part of the Jacobians is rigid but instead the old part (or an isogeny factor thereof) is parametrized by a ball quotient. Three examples of such families are given in Section~6.3 of~\cite{Rohde}. For each of these, we may then again look for a ``twisting factor''~$D$. 

What changes is that there is no longer the restriction that $D \to \mP^1$ is branched over only three points. Further, $D$ no longer needs to be a single CM curve but can be a family of curves. (The argument in the proof of Proposition~\ref{prop:mNabMatch} that gives $N^\prime = 3$ no longer applies.) All that matters is that there are isogeny factors $A$ and~$B$ of the families of Jacobians~$J_C$ and~$J_D$, both with action of~$\mQ(\zeta_d)$, for some $d|m$, such that 
\begin{itemize}
\item[(1)] the desingularization of the surface $T_{\lambda,\mu} = (C_\lambda \times D_\mu)/G$ has $p_\geom = 1$, and for general parameter values $\lambda$ and~$\mu$ the transcendental part of its~$H^2$ is isomorphic to $H^1(A_\lambda) \otimes_{\mQ(\zeta_d)} H^1(B_\mu)$;
\item[(2)] in the moduli space of abelian varieties, the abelian schemes $A$ and~$B$ both trace out a special subvariety, so that there is a dense set of parameter values $\lambda$ and~$\mu$ for which $A_\lambda$ and~$B_\mu$ are of CM type.
\end{itemize}
Once we have found a family of curves~$D_\mu$ for which this holds, we can proceed as before.

Table~\ref{table:further} gives four families of surfaces with $p_\geom = 1$ that were found in this way. In three cases this is a family of K3 surfaces; the fourth example is a family of surfaces of general type with $K^2 = 1$.

\section{Is every K3 surface dominated by a product of curves?}
\label{sec:DPC}

A variety~$X$ is said to be DPC (dominated by a product of curves) if there exists a dominant rational map from a product of curves to~$X$. The surfaces we have discussed in the previous sections are product-quotients and therefore visibly have this property. In particular, this gives a large number of explicit examples of families of surfaces that are DPC; the largest family (no.~150) giving a $9$-dimensional subvariety in the moduli space of K3's. For a discussion of the motivic implications of this (e.g., the finite-dimensionality of the motives of these surfaces) see for instance~\cite{Laterveer}.

In general, it may be hard to decide whether or not some given variety~$X$ is DPC. In a letter to Grothendieck from March~31, 1964, Serre proved that there exist surfaces that are not DPC. (See~\cite{GrothSerre}; we thank Frans Oort for the reference.) Deligne in \cite{DelK3}, Section~7 gave a Hodge-theoretic criterion that can be used to show that varieties are not DPC. These ideas were refined by Schoen in~\cite{Schoen}. For surfaces with $p_\geom = 1$, however, these methods do not lead to an answer, and among the open problems mentioned at the end of~\cite{Schoen} we find the question that is the title of this section. It appears that the expectation among experts is that the answer should be negative. While we are unable to prove this, our main result in this section shows that for a general K3 surface, if it is dominated by a product of curves, the genera of these curves cannot be too small.

\subsection{}
\label{ssec:spinrep}
Let $T$ be a Hodge structure of K3 type with $\dim(T) = 2\ell+1$ for some $\ell \geq 1$ and $\End_\QHS(T) = \mQ$. Up to scalars there is a unique polarization $\phi \colon T \otimes T \to \mQ(0)$, and by the results of Zarhin in~\cite{Zarhin}, the Mumford-Tate group of~$T$ is the special orthogonal group $\SO(\phi)$.

Consider the algebraic group $\CSpin(\phi)$, whose derived subgroup $\Spin(\phi)$ is the simply connected cover of~$\SO(\phi)$. Write $R_\spin$ for the spin representation of $\CSpin(\phi)$ over~$\Qbar$. There is a unique irreducible representation $r_\spin \colon \CSpin(\phi) \to \GL(V)$ over~$\mQ$ such that $r_{\spin,\Qbar} \cong R_\spin^{\oplus m}$ for some $m \geq 1$. More precisely, the even Clifford algebra $C^+(\phi)$ is a central simple $\mQ$-algebra whose Brauer class has order dividing~$2$. If $C^+(\phi)$ is a matrix algebra then $r_\spin$ is an absolutely irreducible representation of dimension~$2^\ell$, so $r_{\spin,\Qbar} \cong R_\spin$; if $C^+(\phi)$ is a matrix algebra over a non-split quaternion algebra then $r_\spin$ has dimension~$2^{\ell+1}$ and $r_{\spin,\Qbar} \cong R_\spin^{\oplus 2}$.

The homomorphism $h \colon \mS \to \SO(\phi)$ that gives the Hodge structure on~$T$ lifts to a homomorphism $\tilde{h} \colon \mS \to \CSpin(\phi)$ such that $r_{\spin,\mR} \circ \tilde{h}$ defines on~$V$ an irreducible polarizable $\mQ$-Hodge structure of type $(0,1) + (1,0)$. The latter Hodge structure is the $H^1$ of a simple abelian variety~$A$ up to isogeny that we refer to as the simple Kuga-Satake partner of~$T$. We have $\dim(A) = 2^{\ell -1}$ or $\dim(A) = 2^\ell$, depending on the order of the class of~$C^+(\phi)$ in the Brauer group.

\begin{theorem}
\label{thm:K3Thm}
Let $S$ be a complex K3 surface which is general in the sense that the transcendental part $T \subset H^2(S,\mQ)\bigl(1\bigr)$ has dimension~$21$ and $\End_\QHS(T) = \mQ$. If $C$ and~$D$ are (complete non-singular) curves such that there exists a dominant rational map from $C \times D$ to~$S$ then the Jacobians $J_C$ and~$J_D$ both contain the simple Kuga-Satake partner of~$T$ as an isogeny factor. In particular, $C$ and~$D$ both have genus at least~$512$.  
\end{theorem}

\begin{proof}
Under the assumptions as stated, there exist simple factors $X$ of~$J_C$ and $Y$ of~$J_D$ such that $T(-1)$ is isomorphic to a sub-Hodge structure of $H_X \otimes H_Y$, where $H_X = H^1(X,\mQ)$ and $H_Y = H^1(Y,\mQ)$. On Mumford-Tate group this gives a surjection
\[
\MT(X \times Y) \twoheadrightarrow \MT\bigl(T(-1)\bigr) = \GO(\phi)\, .
\]
Let $G \triangleleft \MT^\der(X\times Y)$ be the unique $\mQ$-simple factor such that this homomorphism restricts to an isogeny $G \to \SO(\phi)$. Viewing $G$ as a quotient of the spin group~$\Spin(\phi)$, we obtain a representation
\[
\rho \colon \Spin(\phi) \to G \hookrightarrow \MT(X\times Y) \twoheadrightarrow \MT(X) \hookrightarrow \GL(H_X)\, .
\]
As the short root in the Dynkin diagram of type~$\mathrm{B}_\ell$ is the only minuscule weight, the only non-trivial irreducible representation of $\Spin(\phi)$ that may occur in~$\rho_\Qbar$ is the spin representation. (See \cite{DelShimura}, Section~1.3.) It follows that $H_X \cong H_{X,0} \oplus H_{X,1}$ as a representation of~$\Spin(\phi)$, where $H_{X,0} \subset H_X$ is the subspace of invariants under~$\rho$ and $H_{X,1}$ is isomorphic to a sum of copies of the representation~$r_\spin$ described in~\ref{ssec:spinrep} (now restricted to~$\Spin(\phi)$). But $G$ is normal in $\MT(X \times Y)$, so this decomposition is preserved under the action of $\MT(X \times Y)$. By our assumption that $X$ is simple, it follows that either $H_X = H_{X,0}$ or $H_X = H_{X,1}$. 

For $Y$ of course the same argument applies. As $G$ acts on~$T$ via the composition $G \to \SO(\phi) \hookrightarrow \GL(T)$, we find that $T(-1)$ can only be a sub-representation of $H_X \otimes H_Y$ if $H_X = H_{X,1}$ and $H_Y = H_{Y,1}$. In particular, the natural homomorphisms $\Spin(\phi) \to \MT(X)$ and $\Spin(\phi) \to \MT(Y)$ are injective.

Next we show that in fact $\Spin(\phi) \isomarrow \MT^\der(X)$, and likewise for~$Y$. Indeed, if $N$ is a simple factor of~$\MT(X)$ then $N_\mR$ has to have non-compact factors, and these must act non-trivially on~$H_{X,\mR}$. On the other hand, in each irreducible representation of $\MT^\der(X)$ that occurs in~$H_{X,\mC}$ only one of the non-compact real factors can act non-trivially. So the conclusion that $H_X = H_{X,1}$ implies that $\MT^\der(X)$ can have no simple factors other than~$\Spin(\phi)$. 

At this point we know that $\MT(X) = Z \cdot \Spin(\phi) \subset \GL(H_X)$, with $Z$ the connected center. The Hodge structure on~$H_X$ gives rise to a cocharacter $\mu \colon \mG_\mult \to \MT(X)_\mC$. We can lift this to a fractional homomorphism from $\mG_\mult$ to $Z_\mC \times \Spin(\phi)_\mC$, as in \cite{DelShimura}, Section~1.3.4, and we know that in the action of~$\mG_\mult$ on~$H_{X,\mC}$ only the weights $-1$ and~$0$ occur. (We normalize Hodge structures by the rule that $z \in \mC^* = \mS(\mR)$ acts on $H^{p,q}$ as multiplication by $z^{-p}\cdot \bar{z}^{-q}$.) On the other hand, in the spin representation the weights $-1/2$ and $1/2$ occur (cf.\ ibid, Table~1.3.9). The conclusion, therefore, is that the fractional map from~$\mG_\mult$ to~$Z$ lands in the scalar multiples of the identity on~$H_{X,\mC}$. By definition of the Mumford-Tate group it follows that $\MT(X) = \mG_\mult \cdot \Spin(\phi)$, and because $X$ is simple $H_X \cong r_\spin$ as a representation of~$\Spin(\phi)$. The same conclusions holds for~$Y$. As the representations of $\Spin(\phi)$ on~$H_X$ and~$H_Y$ are the same, we conclude that the two projections
\[
\MT(X) \xleftarrow{~\pr_1~} \MT(X\times Y) \xrightarrow{~\pr_2~} \MT(Y)
\]
are isomorphisms. Now we find that the homomorphism $\mS \to \MT(X\times Y)_\mR$ that defines the Hodge structure on $H_X \oplus H_Y$ is the unique lift~$\tilde{h}$ as in~\ref{ssec:spinrep}, and we conclude that $X \cong Y$ (up to isogeny) is the simple Kuga-Satake partner of~$X$.
\end{proof}

\begin{remark}
Similar arguments can be used to extend the result to K3 surfaces with higher Picard number, or even to the situation where the endomorphism field of the transcendental part of the $H^2$ is a totally real field bigger than~$\mQ$. The details get more involved, though. Such a more general result may also be applied to other classes of surfaces with $p_\geom = 1$.
\end{remark}

\titleformat{\section}[block]
{\filcenter\normalfont\large\bfseries}{\appendixname~\thesection.}{.5em}{}

\begin{appendices}

\section{Overview of the examples}
\label{sec:appendix}

\subsection{}
In Table~\ref{table:mNab} we present the list of one hundred and fifty $4$-tuples $(m,N,a,b)$ as in~\ref{ssec:mNab} that serve as input for Theorems~\ref{thm:TsGiveSpecial} and~\ref{thm:MainThm}.

The columns in this table have the following meaning. The first column is just a reference number. Columns 3--6 give the data $(m,N,a,b)$. In case the triple $(m,N,a)$ also appears in the list at the end of~\cite{Mostow}, the second column contains the reference number of that table. (Note that the same $(m,N,a)$ may appear more than once, and for some $(m,N,a)$ in Mostow's table there is no matching~``$b$'', so that it does not occur in our table. In Mostow's table only triples with $N \geq 5$ occur.) Columns 7 and~8 give the genera of the curves~$C_\lambda$ and~$D$. The last column gives the value of~$K^2$ for the minimal model~$S^{\min}$.

\subsection{}
For $N \geq 5$, Deligne and Mostow have obtained a list of data $(m,N,a)$ such that, with notation as in Section~\ref{sec:ExaDMList}, the monodromy of the variation~$\cH^\new$ is an arithmetic subgroup of~$\UU(\mR)$. Let us note that the perspective taken in \cite{DelMost} and~\cite{Mostow} is different from ours; in particular, the non-arithmetic examples found in these papers are not relevant for us. 

In almost all our examples (at least for $N \geq 5$), the first three items $(m,N,a)$ form a triple in the Deligne-Mostow list. There are some examples (102 and 113) where this is not true; the point is that these do not satisfy the condition $\Sigma$INT that appears in the work of Deligne and Mostow and that is irrelevant for us.

For $N=4$ the situation is more delicate. Deligne and Mostow in this case refer to Takeuchi's classification of arithmetic triangle groups, and explain how Takeuchi's parameters $(p,q,r)$ can be transformed into a triple $(m,4,a)$. (See \cite{DelMost}, Section~14.3, and write $\mu_i = a_i/m$.) In most cases, however, the triple thus obtained does not satisfy the arithmeticity condition in ibid., Prop.~12.7. This is not a contradiction; the point is that the triangle group is not an arithmetic subgroup in a unitary group over a cyclotomic field, but in a quaternionic form of a unitary group.

The list we present here was found by a ``brute force'' computer search. We do not claim it is complete.

\vspace{8mm}

\newcounter{rowno}
\setcounter{rowno}{0}

\rowcolors*{2}{gray!20!white!80}{white}
\begin{longtable}{|>{\stepcounter{rowno}\therowno.}r r >{$}c<{$} >{$}c<{$} >{$}c<{$} >{$}c<{$} >{$}c<{$} >{$}c<{$} >{$}c<{$}|}
\caption{\textbf{Examples\label{table:mNab}}}\\
\hline
\hiderowcolors
\multicolumn{1}{|c}{No.} & \multicolumn{1}{c}{DM No.} & \; m\;{} & \; N\;{} & a & b & \; g_C\;{} & \; g_D\;{} & \quad K^2\quad{} \\
\showrowcolors
\endfirsthead
\hline
\hiderowcolors
\multicolumn{1}{|c}{No.} & \multicolumn{1}{c}{DM No.} & m & N & a & b & g_C & g_D & K^2 \\ 
\hline
\showrowcolors
\endhead
\hline
\endfoot
\hline
& & 3 & 4 & (1, 1, 2, 2) & (1, 1, 1) & 2 & 1 & 0 \\
& & 4 & 4 & (1, 1, 3, 3) & (1, 1, 2) & 3 & 1 & 0 \\
& & 4 & 4 & (1, 2, 2, 3) & (1, 1, 2) & 2 & 1 & 0 \\
& & 5 & 4 & (1, 3, 3, 3) & (1, 1, 3) & 4 & 2 & 0 \\
& & 5 & 4 & (2, 2, 2, 4) & (1, 2, 2) & 4 & 2 & 0 \\
& & 6 & 4 & (1, 1, 5, 5) & (1, 2, 3) & 5 & 1 & 0 \\
& & 6 & 4 & (1, 2, 4, 5) & (1, 2, 3) & 4 & 1 & 0 \\
& & 6 & 4 & (1, 3, 3, 5) & (1, 1, 4) & 3 & 2 & 0 \\
& & 6 & 4 & (1, 3, 3, 5) & (1, 2, 3) & 3 & 1 & 0 \\
& & 6 & 4 & (1, 3, 4, 4) & (1, 1, 4) & 3 & 2 & 0 \\
& & 6 & 4 & (1, 3, 4, 4) & (1, 2, 3) & 3 & 1 & 0 \\
& & 6 & 4 & (2, 2, 3, 5) & (1, 2, 3) & 3 & 1 & 0 \\
& & 6 & 4 & (2, 3, 3, 4) & (1, 1, 4) & 2 & 2 & 0 \\
& & 6 & 4 & (2, 3, 3, 4) & (1, 2, 3) & 2 & 1 & 0 \\
& & 7 & 4 & (2, 4, 4, 4) & (1, 2, 4) & 6 & 3 & 0 \\
& & 7 & 4 & (3, 3, 3, 5) & (1, 3, 3) & 6 & 3 & 0 \\
& & 8 & 4 & (1, 3, 6, 6) & (1, 3, 4) & 6 & 2 & 0 \\
& & 8 & 4 & (1, 5, 5, 5) & (1, 2, 5) & 7 & 3 & 0 \\
& & 8 & 4 & (2, 4, 5, 5) & (1, 2, 5) & 5 & 3 & 0 \\
& & 8 & 4 & (3, 3, 3, 7) & (1, 3, 4) & 7 & 2 & 0 \\
& & 8 & 4 & (3, 3, 3, 7) & (2, 3, 3) & 7 & 3 & 0 \\
& & 8 & 4 & (3, 3, 4, 6) & (1, 3, 4) & 5 & 2 & 0 \\
& & 8 & 4 & (3, 3, 4, 6) & (2, 3, 3) & 5 & 3 & 0 \\
& & 9 & 4 & (1, 5, 5, 7) & (1, 3, 5) & 8 & 3 & 0 \\
& & 9 & 4 & (2, 4, 4, 8) & (2, 3, 4) & 8 & 3 & 0 \\
& & 9 & 4 & (3, 5, 5, 5) & (1, 3, 5) & 7 & 3 & 0 \\
& & 9 & 4 & (4, 4, 4, 6) & (1, 4, 4) & 7 & 4 & 0 \\
& & 9 & 4 & (4, 4, 4, 6) & (2, 3, 4) & 7 & 3 & 0 \\
& & 10 & 4 & (1, 4, 7, 8) & (1, 4, 5) & 8 & 2 & 0 \\
& & 10 & 4 & (1, 5, 7, 7) & (1, 2, 7) & 7 & 4 & 0 \\
& & 10 & 4 & (1, 5, 7, 7) & (1, 4, 5) & 7 & 2 & 0 \\
& & 10 & 4 & (2, 3, 6, 9) & (2, 3, 5) & 8 & 2 & 0 \\
& & 10 & 4 & (2, 4, 7, 7) & (1, 4, 5) & 8 & 2 & 0 \\
& & 10 & 4 & (3, 3, 5, 9) & (2, 3, 5) & 7 & 2 & 0 \\
& & 10 & 4 & (3, 3, 5, 9) & (3, 3, 4) & 7 & 4 & 0 \\
& & 10 & 4 & (3, 3, 6, 8) & (2, 3, 5) & 8 & 2 & 0 \\
& & 10 & 4 & (3, 5, 6, 6) & (1, 3, 6) & 6 & 4 & 0 \\
& & 10 & 4 & (3, 5, 6, 6) & (2, 3, 5) & 6 & 2 & 0 \\
& & 10 & 4 & (4, 4, 5, 7) & (1, 4, 5) & 6 & 2 & 0 \\
& & 12 & 4 & (1, 5, 8, 10) & (1, 5, 6) & 9 & 3 & 0 \\
& & 12 & 4 & (1, 7, 7, 9) & (1, 4, 7) & 10 & 4 & 0 \\
& & 12 & 4 & (2, 6, 7, 9) & (1, 4, 7) & 7 & 4 & 1 \\
& & 12 & 4 & (2, 7, 7, 8) & (2, 3, 7) & 9 & 4 & 0 \\
& & 12 & 4 & (3, 5, 5, 11) & (3, 4, 5) & 10 & 3 & 0 \\
& & 12 & 4 & (3, 5, 6, 10) & (3, 4, 5) & 7 & 3 & 0 \\
& & 12 & 4 & (3, 5, 8, 8) & (1, 5, 6) & 7 & 3 & 1 \\
& & 12 & 4 & (3, 7, 7, 7) & (1, 4, 7) & 10 & 4 & 0 \\
& & 12 & 4 & (3, 7, 7, 7) & (2, 3, 7) & 10 & 4 & 0 \\
& & 12 & 4 & (4, 5, 5, 10) & (1, 5, 6) & 9 & 3 & 0 \\
& & 12 & 4 & (4, 5, 5, 10) & (3, 4, 5) & 9 & 3 & 0 \\
& & 12 & 4 & (4, 6, 7, 7) & (1, 4, 7) & 7 & 4 & 0 \\
& & 12 & 4 & (4, 6, 7, 7) & (2, 3, 7) & 7 & 4 & 0 \\
& & 12 & 4 & (5, 5, 5, 9) & (1, 5, 6) & 10 & 3 & 0 \\
& & 12 & 4 & (5, 5, 5, 9) & (2, 5, 5) & 10 & 5 & 0 \\
& & 12 & 4 & (5, 5, 5, 9) & (3, 4, 5) & 10 & 3 & 0 \\
& & 12 & 4 & (5, 5, 6, 8) & (1, 5, 6) & 7 & 3 & 0 \\
& & 12 & 4 & (5, 5, 6, 8) & (2, 5, 5) & 7 & 5 & 0 \\
& & 12 & 4 & (5, 5, 6, 8) & (3, 4, 5) & 7 & 3 & 0 \\
& & 14 & 4 & (1, 5, 11, 11) & (2, 5, 7) & 13 & 3 & 2 \\
& & 14 & 4 & (1, 9, 9, 9) & (3, 4, 7) & 13 & 3 & 3 \\
& & 14 & 4 & (2, 5, 10, 11) & (2, 5, 7) & 12 & 3 & 0 \\
& & 14 & 4 & (3, 3, 9, 13) & (3, 4, 7) & 13 & 3 & 0 \\
& & 14 & 4 & (3, 4, 9, 12) & (3, 4, 7) & 12 & 3 & 0 \\
& & 14 & 4 & (3, 7, 9, 9) & (2, 3, 9) & 10 & 6 & 0 \\
& & 14 & 4 & (3, 7, 9, 9) & (3, 4, 7) & 10 & 3 & 0 \\
& & 14 & 4 & (4, 6, 9, 9) & (3, 4, 7) & 12 & 3 & 0 \\
& & 14 & 4 & (5, 5, 5, 13) & (2, 5, 7) & 13 & 3 & 0 \\
& & 14 & 4 & (5, 5, 7, 11) & (2, 5, 7) & 10 & 3 & 0 \\
& & 14 & 4 & (5, 5, 7, 11) & (4, 5, 5) & 10 & 6 & 0 \\
& & 14 & 4 & (5, 5, 8, 10) & (2, 5, 7) & 12 & 3 & 0 \\
& & 15 & 4 & (2, 8, 8, 12) & (2, 5, 8) & 13 & 5 & 0 \\
& & 15 & 4 & (3, 7, 7, 13) & (3, 5, 7) & 13 & 4 & 0 \\
& & 15 & 4 & (4, 8, 8, 10) & (3, 4, 8) & 12 & 6 & 0 \\
& & 15 & 4 & (5, 7, 7, 11) & (3, 5, 7) & 12 & 4 & 0 \\
& & 15 & 4 & (6, 8, 8, 8) & (2, 5, 8) & 13 & 5 & 0 \\
& & 15 & 4 & (7, 7, 7, 9) & (3, 5, 7) & 13 & 4 & 0 \\
& & 16 & 4 & (7, 7, 7, 11) & (1, 7, 8) & 15 & 4 & 0 \\
& & 18 & 4 & (4, 5, 12, 15) & (4, 5, 9) & 13 & 4 & 0 \\
& & 18 & 4 & (5, 5, 11, 15) & (4, 5, 9) & 16 & 4 & 0 \\
& & 18 & 4 & (7, 7, 7, 15) & (2, 7, 9) & 16 & 4 & 0 \\
& & 18 & 4 & (7, 7, 10, 12) & (2, 7, 9) & 14 & 4 & 0 \\
& & 22 & 4 & (7, 7, 13, 17) & (4, 7, 11) & 21 & 5 & 0 \\
& & 22 & 4 & (9, 9, 9, 17) & (2, 9, 11) & 21 & 5 & 0 \\
& 41. & 3 & 5 & (1, 1, 1, 1, 2) & (1, 1, 1) & 3 & 1 & 0 \\
& 42. & 4 & 5 & (1, 1, 1, 2, 3) & (1, 1, 2) & 4 & 1 & 0 \\
& 43. & 4 & 5 & (1, 1, 2, 2, 2) & (1, 1, 2) & 3 & 1 & 0 \\
& 44. & 5 & 5 & (2, 2, 2, 2, 2) & (1, 2, 2) & 6 & 2 & 0 \\
& 45.& 6 & 5 & (1, 1, 1, 4, 5) & (1, 2, 3) & 7 & 1 & 0 \\
& 46. & 6 & 5 & (1, 1, 2, 3, 5) & (1, 2, 3) & 6 & 1 & 0 \\
& 47. & 6 & 5 & (1, 1, 2, 4, 4) & (1, 2, 3) & 6 & 1 & 0 \\
& 48. & 6 & 5 & (1, 1, 3, 3, 4) & (1, 1, 4) & 5 & 2 & 0 \\
& 48. & 6 & 5 & (1, 1, 3, 3, 4) & (1, 2, 3) & 5 & 1 & 0 \\
& 49. & 6 & 5 & (1, 2, 2, 2, 5) & (1, 2, 3) & 6 & 1 & 0 \\
& 50. & 6 & 5 & (1, 2, 2, 3, 4) & (1, 2, 3) & 5 & 1 & 0 \\
& 51. & 6 & 5 & (1, 2, 3, 3, 3) & (1, 1, 4) & 4 & 2 & 1 \\
& 51. & 6 & 5 & (1, 2, 3, 3, 3) & (1, 2, 3) & 4 & 1 & 0 \\
& 52. & 6 & 5 & (2, 2, 2, 3, 3) & (1, 2, 3) & 4 & 1 & 0 \\
& 53. & 8 & 5 & (1, 3, 3, 3, 6) & (1, 3, 4) & 10 & 2 & 0 \\
& 55. & 8 & 5 & (3, 3, 3, 3, 4) & (1, 3, 4) & 9 & 2 & 0 \\
& 55. & 8 & 5 & (3, 3, 3, 3, 4) & (2, 3, 3) & 9 & 3 & 0 \\
& 56. & 9 & 5 & (2, 4, 4, 4, 4) & (2, 3, 4) & 12 & 3 & 0 \\
& & 10 & 5 & (1, 1, 4, 7, 7) & (1, 4, 5) & 13 & 2 & 0 \\
& 57. & 10 & 5 & (1, 4, 4, 4, 7) & (1, 4, 5) & 12 & 2 & 0 \\
& 58. & 10 & 5 & (2, 3, 3, 3, 9) & (2, 3, 5) & 13 & 2 & 0 \\
& 59. & 10 & 5 & (2, 3, 3, 6, 6) & (2, 3, 5) & 12 & 2 & 0 \\
& 60. & 10 & 5 & (3, 3, 3, 3, 8) & (2, 3, 5) & 13 & 2 & 0 \\
& 61. & 10 & 5 & (3, 3, 3, 5, 6) & (2, 3, 5) & 11 & 2 & 0 \\
& 62. & 12 & 5 & (1, 5, 5, 5, 8) & (1, 5, 6) & 15 & 3 & 0 \\
& 72. & 12 & 5 & (3, 5, 5, 5, 6) & (3, 4, 5) & 13 & 3 & 0 \\
& 75. & 12 & 5 & (4, 5, 5, 5, 5) & (1, 5, 6) & 15 & 3 & 0 \\
& 75. & 12 & 5 & (4, 5, 5, 5, 5) & (3, 4, 5) & 15 & 3 & 0 \\
& 76. & 14 & 5 & (2, 5, 5, 5, 11) & (2, 5, 7) & 19 & 3 & 0 \\
& & 14 & 5 & (3, 3, 4, 9, 9) & (3, 4, 7) & 19 & 3 & 0 \\
& 77. & 14 & 5 & (5, 5, 5, 5, 8) & (2, 5, 7) & 19 & 3 & 0 \\
& 23. & 3 & 6 & (1, 1, 1, 1, 1, 1) & (1, 1, 1) & 4 & 1 & 0 \\
& 24. & 4 & 6 & (1, 1, 1, 1, 1, 3) & (1, 1, 2) & 6 & 1 & 0 \\
& 25. & 4 & 6 & (1, 1, 1, 1, 2, 2) & (1, 1, 2) & 5 & 1 & 0 \\
& 26. & 6 & 6 & (1, 1, 1, 1, 3, 5) & (1, 2, 3) & 9 & 1 & 0 \\
& 27. & 6 & 6 & (1, 1, 1, 1, 4, 4) & (1, 2, 3) & 9 & 1 & 0 \\
& 28. & 6 & 6 & (1, 1, 1, 2, 2, 5) & (1, 2, 3) & 9 & 1 & 0 \\
& 29. & 6 & 6 & (1, 1, 1, 2, 3, 4) & (1, 2, 3) & 8 & 1 & 0 \\
& 30. & 6 & 6 & (1, 1, 1, 3, 3, 3) & (1, 1, 4) & 7 & 2 & 1 \\
& 30. & 6 & 6 & (1, 1, 1, 3, 3, 3) & (1, 2, 3) & 7 & 1 & 0 \\
& 31.  & 6 & 6 & (1, 1, 2, 2, 2, 4) & (1, 2, 3) & 8 & 1 & 0 \\
& 32. & 6 & 6 & (1, 1, 2, 2, 3, 3) & (1, 2, 3) & 7 & 1 & 0 \\
& 33. & 6 & 6 & (1, 2, 2, 2, 2, 3) & (1, 2, 3) & 7 & 1 & 0 \\
& 34. & 8 & 6 & (1, 3, 3, 3, 3, 3) & (1, 3, 4) & 14 & 2 & 0 \\
& 35. & 10 & 6 & (2, 3, 3, 3, 3, 6) & (2, 3, 5) & 17 & 2 & 0 \\
& 36. & 10 & 6 & (3, 3, 3, 3, 3, 5) & (2, 3, 5) & 16 & 2 & 0 \\
& 14. & 4 & 7 & (1, 1, 1, 1, 1, 1, 2) & (1, 1, 2) & 7 & 1 & 0 \\
& 15. & 6 & 7 & (1, 1, 1, 1, 1, 2, 5) & (1, 2, 3) & 12 & 1 & 0 \\
& 16. & 6 & 7 & (1, 1, 1, 1, 1, 3, 4) & (1, 2, 3) & 11 & 1 & 0 \\
& 17. & 6 & 7 & (1, 1, 1, 1, 2, 2, 4) & (1, 2, 3) & 11 & 1 & 0 \\
& 18. & 6 & 7 & (1, 1, 1, 1, 2, 3, 3) & (1, 2, 3) & 10 & 1 & 0 \\
& 19. & 6 & 7 & (1, 1, 1, 2, 2, 2, 3) & (1, 2, 3) & 10 & 1 & 0 \\
& 20. & 6 & 7 & (1, 1, 2, 2, 2, 2, 2) & (1, 2, 3) & 10 & 1 & 0 \\
& 21. & 10 & 7 & (2, 3, 3, 3, 3, 3, 3) & (2, 3, 5) & 22 & 2 & 0 \\
& 8. & 4 & 8 & (1, 1, 1, 1, 1, 1, 1, 1) & (1, 1, 2) & 9 & 1 & 0 \\
& 9. & 6 & 8 & (1, 1, 1, 1, 1, 1, 1, 5) & (1, 2, 3) & 15 & 1 & 0 \\
& 10. & 6 & 8 & (1, 1, 1, 1, 1, 1, 2, 4) & (1, 2, 3) & 14 & 1 & 0 \\
& 11. & 6 & 8 & (1, 1, 1, 1, 1, 2, 2, 3) & (1, 2, 3) & 13 & 1 & 0 \\
& 12. & 6 & 8 & (1, 1, 1, 1, 1, 1, 3, 3) & (1, 2, 3) & 13 & 1 & 0 \\
& 13. & 6 & 8 & (1, 1, 1, 1, 2, 2, 2, 2) & (1, 2, 3) & 13 & 1 & 0 \\
& 5. & 6 & 9 & (1, 1, 1, 1, 1, 1, 1, 1, 4) & (1, 2, 3) & 17 & 1 & 0 \\
& 6. & 6 & 9 & (1, 1, 1, 1, 1, 1, 1, 2, 3) & (1, 2, 3) & 16 & 1 & 0 \\
& 7. & 6 & 9 & (1, 1, 1, 1, 1, 1, 2, 2, 2) & (1, 2, 3) & 16 & 1 & 0 \\
& 3. & 6 & 10 & (1, 1, 1, 1, 1, 1, 1, 1, 1, 3) & (1, 2, 3) & 19 & 1 & 0 \\
& 4. & 6 & 10 & (1, 1, 1, 1, 1, 1, 1, 1, 2, 2) & (1, 2, 3) & 19 & 1 & 0 \\
& 2. & 6 & 11 & (1, 1, 1, 1, 1, 1, 1, 1, 1, 1, 2) & (1, 2, 3) & 22 & 1 & 0 \\
& 1. & 6 & 12 & (1, 1, 1, 1, 1, 1, 1, 1, 1, 1, 1, 1) & (1, 2, 3) & 25 & 1 & 0 \\
\hline
\end{longtable}

\vspace{8mm}

\subsection{}
Table~\ref{table:further} gives the building data for the four families mentioned in~\ref{ssec:furtherexa}. In this table we no longer list the length of the string~$a$ (which is redundant information anyway). The first three of these examples give a $2$-dimensional family of surfaces, the last a $1$-dimensional family.

\vspace{8mm}

\rowcolors*{2}{gray!20!white!80}{white}
\begin{longtable}{|>{\stepcounter{rowno}\therowno.}r >{$}c<{$} >{$}c<{$} >{$}c<{$} >{$}c<{$} >{$}c<{$} >{$}c<{$}|}
\caption{\textbf{Further examples\label{table:further}}}\\
\hline
\hiderowcolors
\multicolumn{1}{|c}{No.}  & \; m\;{} & a & b & \; g_C\;{} & \; g_D\;{} & \quad K^2\quad{} \\
\showrowcolors
\endfirsthead
\hline
\hiderowcolors
\multicolumn{1}{|c}{No.} & m & a & b & g_C & g_D & K^2 \\ 
\hline
\showrowcolors
\endhead
\hline
\endfoot
\hline
& 4 & (1, 1, 1, 1) & (1, 1, 1, 1) & 3 & 3 & 0 \\
& 6 & (1, 1, 1, 3) & (1, 1, 1, 3) & 4 & 4 & 0 \\
& 6 & (1, 1, 1, 3) & (1, 1, 2, 2) & 4 & 4 & 1 \\
& 6 & (1, 1, 2, 2) & (1, 1, 4) & 4 & 2 & 0 \\
\hline
\end{longtable}

\end{appendices}

{\small

\bigskip

} 

\noindent
\texttt{b.moonen@science.ru.nl}

\noindent
Radboud University Nijmegen, IMAPP, Nijmegen, The Netherlands

\end{document}